\documentclass[leqno,12pt]{amsart}
\usepackage{amscd,amsfonts,amsmath,amssymb,amstext,amsthm}
\newtheorem{theorem}{Theorem}
\newtheorem{lemma}{Lemma}[section]
\newtheorem{proposition}{Proposition}[section]
\newtheorem{corollary}{Corollary}[section]
\newtheorem{definition}{\it Definiton}[section]
\newtheorem{example}{\it Example}[section]
\numberwithin{equation}{section}
\begin{document}
\title{Average number of zeros and mixed symplectic volume of Finsler sets}
\author{Dmitri Akhiezer and Boris Kazarnovskii}
\address {Institute for Information Transmission Problems \newline
19 B.Karetny per.,127994, Moscow, Russia,\newline
{\rm D.A.:} 
{\it akhiezer@iitp.ru},  
{\rm B.K.:} 
{\it kazbori@gmail.com}.}
\thanks{The research was carried out at the Institute for Information Transmission Problems under support by 
the Russian Foundation of Sciences, 
grant  No.14-50-00150. }
\begin{abstract}
Let $X$ be an $n$-dimensional manifold and $V _1, \ldots, V_n \subset C^\infty(X, \mathbb R)$  finite-dimensional vector spaces
with Euclidean metric. We assign to each $V_i$ a Finsler ellipsoid, i.e., a family of ellipsoids in the fibers
of the cotangent bundle to $X$. We prove that the average number of isolated common zeros of $f_1\in V_1, \ldots, f_n \in V_n$
is equal to the mixed symplectic volume of these Finsler ellipsoids.
If $X$ is a homogeneous space of a compact Lie group and all vector spaces $V_i$ together with their Euclidean metrics are invariant,
then the average numbers of zeros satisfy the inequalities, similar to Hodge
inequalities for intersection numbers of divisors on a projective variety. This is applied to the eigenspaces of Laplace
operator of an invariant Riemannian metric. The proofs are based on a construction of the ring of normal
densities on $X$, an analogue of the ring of differential forms. In particular, this construction is used to carry over
the Crofton formula to the product of spheres. 
\end{abstract}
\renewcommand{\subjclassname}
{\textup{2010} Mathematics Subject Classification}
\maketitle
\renewcommand{\thefootnote}{}
\maketitle

\section{Introduction}\label{intro}

Let $X$ be a differentiable manifold of dimension $n$. We assume that $X$ is connected and has
countable base of topology. Let
 $V \subset C^\infty(X, {\mathbb R})$ be a finite-dimensional vector space with Euclidean metric, such that
\begin{equation}\label{value}
\forall x \in X \ \exists f\in V: f(x) \ne 0.
\end{equation}
By a Finsler set or an $F$-set we mean a continuous family ${\mathcal E}=\{{\mathcal E}(x)\}$
of compact convex sets ${\mathcal E}(x) \subset T_x^*$.
For a given $V$ we construct the $F$-set ${\mathcal E} =\{{\mathcal E}(x)\}$, in which all ${\mathcal E}(x)$
are ellipsoids in the corresponding cotangent spaces. It turns out that the average number
of isolated common zeros of $f_1, \ldots, f_n \in V$, defined below in \ref{average}, equals the symplectic volume
of the domain 
\begin{equation} \label{domain}
\bigcup_{x\in X}{\mathcal E}(x) \subset T^*X, 
\end{equation}
multiplied by a constant depending on $n$.

This is easily understood in the special case of a submanifold $X$ of the unit sphere $S^{N-1} \subset {\mathbb R}^N$,
provided $X$ is not contained in a proper vector subspace and $V$ consists of linear functionals on ${\mathbb R}^N$.
Namely, give $X$ a Riemannian metric induced by the Euclidean metric of ${\mathbb R}^N$. Then
Crofton formula for the sphere tells us that the average number of zeros
is proportional to the volume of $X$. More precisely, this number is equal to ${2\over \sigma _n}{\rm vol }(X)$, where $\sigma _n$
is the volume of the $n$-dimensional unit sphere,
see e.g. \cite{Sa}. On the other hand, the Riemannian metric on $X$ allows us to identify $T_x^* $ and $T_x$.
If the ellipsoid ${\mathcal E}(x) \subset T_x^*$ is defined as the unit ball in this metric then the symplectic volume of the domain 
(\ref{domain}) differs from ${\rm vol}(X)$ by a coefficient depending only on $n$.

It is not hard to prove the result on the isolated common zeros of $f_1, \ldots, f_n \in V$ even if $X$ is not embedded in $S^{N-1}$.
The main difficulty appears in the case of $n$ Euclidean vector spaces
$V_1, \ldots, V_n \subset C^\infty(X,{\mathbb R})$. Then we have $n$ $F$-sets, namely,
$F$-ellipsoids ${\mathcal E}_1, \ldots, {\mathcal E}_n$
corresponding to $V_1, \ldots, V_n$. We define the mixed volume of $F$-sets and prove Theorem \ref{first}, showing that the average number of isolated zeros of $f_1 \in V_1, \ldots,
f_n \in V_n$ equals  the mixed volume of ${\mathcal E}_1, \ldots, {\mathcal E}_n$.
Theorem \ref{first} can be viewed as a real geometric counterpart of Bernstein-Kouchnirenko theorem,
relating the number of zeros of certain algebraic equations with the mixed volume of attached convex polytopes, see
\cite{Be}. The mixed volumes of convex polytopes satisfy  
Alexandrov-Fenchel inequalities, which yield
Hodge inequalities for intersection indices on some algebraic varieties, e.g., in the toric case \cite{Kh};
see also \cite{KK} for generalizations. In a special situation, we prove inequalities of this type, which we again call  Hodge inequalities,
for the average numbers of isolated zeros.  
Namely, we prove them for homogeneous spaces of compact Lie groups
and invariant Euclidean spaces $V_i$, see Theorem \ref{H-inequalities}.
Here too they are a consequence of Alexandrov-Fenchel inequalities.

The proof of Theorem \ref{first} is based on two facts. One of them is
Theorem \ref{prod}, carrying over Crofton formula to the product of spheres. The other one is
Theorem \ref{second},
calculating the product of some special 1-densities on $X$. Recall
that a $k$-density on a vector space 
is a continuous function $\delta $ on the cone of decomposable $k$-vectors, such that
$\delta (t\xi) = \vert t \vert \delta (\xi)$. A $k$-density on a manifold $X$ is a $k$-density
$\delta _x$ on each tangent space $T_x$, such that the assignment $x \mapsto \delta_x$ is
continuous. The main property of $k$-densities lies in the fact that they can be integrated along arbitrary,
not necessarily oriented, submanifolds of dimension $k$. We refer the reader to
\cite{GS}, \cite{PF} for other properties and applications of densities.

If $X$ is equipped with a non-negative quadratic form $g$ then for any $k \le n$ we have a $k$-density ${\rm vol}_{k,g}$.
The value of ${\rm vol}_{k,g}$ on a $k$-vector $\xi _1 \wedge \ldots \wedge \xi _k$ is the $g$-volume of the
parallelotope with edges $\xi _i$. An arbitrary $k$-density assigns to such a parallelotope its ``$\delta $-volume"
$\delta(\xi_1\wedge \ldots \wedge \xi_k)$. In Section \ref {ring} we define a graded subspace ${\mathfrak n}(X)$
in the space of all densities. The elements of ${\mathfrak n}(X)$ are called normal densities. For normal
densities we define the product making
${\mathfrak n}(X)$
into a commutative graded ring. Any smooth 1-density is normal. As an example, we have
the formula ${\rm vol}_{1,g}^k = c(n,k){\rm vol}_{k,g}$.

''Crofton formula for the product of spheres" is obtained following the pattern of complex projective
spaces \cite{Ka1}, \cite{Ka2}. In that case $V_i$ are Hermitian vector spaces of holomorphic functions
on a complex manifold $X$. Let $\theta _i : X \to V_i^*$ be the mapping assigning to $x \in X$ the functional $\theta _i(x)(f) = f(x),\, f\in V_i$. Then the average number of isolated common zeros of $f_i \in V_i$ in a domain $U \subset X$ equals
$\int _U \omega_1 \wedge \ldots \wedge \omega _n$, where $\omega _i$ is the
pull-back of the Fubini-Study form on ${\bf P}(V^*_i)$ under the mapping $X {\buildrel \theta _i \over \longrightarrow}V_i^*
 \to {\bf P}(V_i^*)$.

For a real manifold $X$ we prove that the average number of isolated common zeros of $f_i \in V_i$ in $U \subset X$ is
obtained by integrating over $U$ a certain $n$-density $\Omega $. Namely, if all $V_i$ are subject to (\ref {value})
then we have a similar mapping
$\theta _i $ from $X$ to the sphere in the dual space $V_i^*$ and we can define $g_i$ as the pull-back of the
metric form on that sphere. Then $\Omega = {1\over \pi^n} {\rm vol}_{1,g_1}\cdot \ldots \cdot {\rm vol}_{1,g_n}$,
where ${\rm vol}_{1,g_i}$ is the 1-density corresponding to the non-negative quadratic form $g_i$.
Thus the product of K\"ahler forms is replaced by the product of 1-densities
of the corresponding metrics.

For the proof of Crofton formula we use the standard technique of double fibrations and write $\Omega $ as the pull-back
and push-forward of the volume form on the space of systems of equations, see \cite{GS}, \cite{Ka2}, \cite{PF}, \cite{Sh}.
The main part of the proof is the presentation of $n$-density $\Omega $ as the product of the abovementioned 1-densities,
see Theorem \ref{prod}. 

The deduction of Theorem \ref{first} from Crofton formula relies on the following theorem from convex geometry
(Theorem \ref {A_1,...,A_k}). Given $k$ compact convex sets $A_1, \ldots, A_k \subset {\mathbb R}^n$, define
a $k$-density $d_k(A_1, \ldots, A_k)$ on ${\mathbb R}^n$ by 
$$d_k(A, \ldots, A_k) (\xi_1 \wedge \ldots \wedge \xi_k) = {\rm V}_k(\pi_HA_1, \ldots, \pi _HA_k)\cdot {\rm vol}_k\Pi_\xi,$$
where $\xi _i $ are linearly independent vectors generating parallelotope $\Pi _\xi$,
$\pi _H$ is the orthogonal projection map onto $H ={\mathbb R}\xi _1 + \ldots + {\mathbb R}\xi _k$
and ${\rm V}_k$ is the $k$-dimensional mixed volume. Then for a wide class of centrally symmetric convex sets,
including smooth convex bodies and zonoids (see \ref{products}), one has
\begin{equation}\label{cdot}
d_1(A_1)\cdot \ldots \cdot d_1(A_k) = k!\,d_k(A_1, \ldots, A_k).
\end{equation}

Given a compact convex body in ${\mathbb R}^n$, its $k$-th brightness function is the function
on ${\rm Gr}(k, {\mathbb R}^n)$ equal to the $k$-volume of the orthogonal projection of $A$
on a given $k$-dimensional subspace $H$. For $k =1 $ this function is called the width function. 
Determination of $A$ by its brightness functions and relations between these functions are
studied in geometric tomography, see \cite{Gar},\cite{Go},\cite{Ho}. The $k$-density $d_k(A)$
evaluated on a unit $k$-vector $\xi_1\wedge\ldots \wedge \xi_k $, where $ \xi_i \in H$, is the $k$-brightness of $A$ at $H\in {\rm Gr}(k, {\mathbb R}^n)$.
Therefore (\ref{cdot}) is an assertion from geometric tomography. In particular, for a centrally symmetric
smooth convex body $A$ one has $d_1^k(A) = k!d_k(A)$.
 
The space of normal densities on ${\mathbb R}^n$ can be identified with a subspace of translation invariant
valuations on compact convex sets. Under this identification, the product of smooth normal densities coincides with
Alesker product
of smooth valuations, see \cite{A}. Therefore (\ref{cdot}) can be regarded as an identity from the valuation theory, see \ref{val}.

Theorem \ref{first} and the product theorem for 1-densities related to $F$-sets (Theorem \ref {second}) 
are stated in Section \ref{main results}. Also in this section, we give applications of Theorem \ref{first} to
homogeneous spaces of compact Lie groups. Namely, we deduce Hodge inequalities 
for average numbers of isolated common zeros and consider these numbers in some detail
for eigenfunctions of the Laplace operator of an invariant Riemann metric. In Section \ref{ring} we construct
the ring of normal desities on a vector space and, after that, on a differentiable manifold. Proofs
of Theorems \ref{first} and \ref{second} are given in Section \ref{proofs}.

\noindent
{\sl The authors are grateful to Semyon Alesker for useful discussions.}

\section{Main results}\label{main results}
\subsection{Average number of zeros}\label{average}
Let $X$ be an $n$-dimensional  manifold and let
 $V_1,\ldots, V_n$  be finite dimensional vector subspaces in $C^\infty (X, { \mathbb R})$. Assume that each $V_i$ has a fixed
scalar product $\langle.,.\rangle_i$,  let $S_i \subset V_i$ be the sphere of radius 1
with center 0 and let $\sigma $ be the product of volumes of $S_i$. For a system
of functions $(s_1, \ldots, s_n) \in S_1 \times \ldots \times S_n$ we denote by
$N(s_1, \ldots, s_n)$ the number of isolated common zeros of $s_i,\ i=1,\ldots,n$.
We will see later that the following integral exists. We call 
$${\mathfrak M}_X(V_1, \ldots, V_n) = {1\over \sigma}\int _{S_1\times \ldots \times S_n}N(s_1,\ldots,s_n)\,ds_1\cdot\ldots\cdot ds_n$$
the average number of common zeros of $n$ functions $s_i$. 
For any point $x\in X$ define
the functional $\varphi _i (x) \in V_i^*$ by $ \varphi_i (x)(f) = f(x).$ 
Assuming (\ref{value}) for each $V_i$ we have $\varphi _i (x)\ne 0$. 
Equip $V_i^*$ with the dual scalar product $\langle .,.\rangle^*_i$, denote by $S_i^*$ the unit
sphere in $V_i^*$ and consider the mapping
$\theta _i : X \to S_i^*$ defined by
$$\theta _i(x) = {\varphi_i(x)\over \sqrt {\langle \varphi_i(x), \varphi_i(x)\rangle_i^*}}\,.$$
The pull-back of the Euclidean metric on $S_i^*\subset V_i^*$ under $\theta _i$ is a non-negative
quadratic form $g_i$  on the tangent bundle of $X$. 

Our first theorem computes ${\mathfrak M}_X(V_1, \ldots, V_n)$
in terms of $g_i$. Namely, let $g_{i,x}$ be the quadratic form on the tangent space $T_x$
corresponding to $g_i$.
Note that 
 $\sqrt{g_{i,x}}$ is a convex function and consider the convex set ${\mathcal E}_i(x) \subset T_x^*$ with
support function $\sqrt{g_{i,x}}$. In other words,
$${\rm max}\  _{_{\xi ^*\in {\mathcal E}_i(x)}} \ \xi^*(\xi) = \sqrt{g_{i,x}(\xi )}\,.$$
Then ${\mathcal E}_i(x)$ is a centrally symmetric convex body in the orthogonal
complement to the kernel of $g_{i,x}$. We call ${\mathcal E}_i(x)$ the ellipsoid
associated with $g_i$ at $x\in X$.

Suppose that for every $x \in X$ we are given a compact convex set ${\mathcal E}(x) \subset T_x^*$
depending continuously on $x\in X$. We call the collection ${\mathcal E} =\{{\mathcal E}(x) \ \vert\ x\in X\}$ a Finsler set 
or an $F$-set in $X$. A Finsler set is said to be centrally symmetric
if each ${\mathcal E}(x)$ is centrally symmetric.
In particular, the $F$-set ${\mathcal E}_i = \{{\mathcal E}_i(x)\}$
is centrally symmetric. This $F$-set is called the $F$-ellipsoid associated to $g_i$.

The volume of an $F$-set ${\mathcal E}$ is defined as the volume
of $\cup_{x\in X}{\mathcal E}(x) \subset T^*X$ with respect
to the standard symplectic structure on the cotangent bundle.
More precisely, if the symplectic form is $\omega$ then the volume form is $\omega ^n/ {n!}$.
Using Minkowski sum and homotheties, we consider linear combinations of convex sets with non-negative
coefficients.
The linear combination of $F$-sets is defined by
$$(\sum _i \lambda _i{\mathcal E}_i)(x) = \sum _i \lambda _i {\mathcal E}_i(x).$$
The symplectic volume of the $F$-set $\lambda _1 {\mathcal E}_1 + \ldots +\lambda _n{\mathcal E}_n$
is a homogeneous polynomial of degree $n$ in $\lambda_1, \ldots, \lambda _n$. Its coeficient at $\lambda _1 \cdot \ldots \cdot \lambda _n$ divided by $n!$ is called the mixed volume of $F$-sets ${\mathcal E}_1,
\ldots, {\mathcal E}_n$ and is denoted by ${\rm V}_n^F({\mathcal E}_1, \ldots ,{\mathcal E}_n)$. 

\begin{theorem}\label{first}
Assume that all spaces $V_i$ are subject to {(\ref {value})} and  let ${\mathcal E}_i$ be the $F$-ellipsoid associated to $g_i,
\ i=1,\ldots,n$. Then $${\mathfrak M}_X(V_1, \ldots, V_n) = {n!\over(2\pi)^n}\cdot {\rm V}_n^F({\mathcal E}_1, \ldots ,{\mathcal E}_n)$$ 
\end{theorem}
\noindent
The proof will be given in \ref{proof1}.

Introduce a Riemannian metric $h$ on $X$. Let $h_x$ be the corresponding metric on $T_x$ and $h^*_x$ the
dual metric on $T_x^*$. For $F$-sets ${\mathcal E}_i$ denote by $ {\rm V}_{{\mathcal E}_1, \ldots , {\mathcal E}_n}(x)$
the mixed volume of the convex sets ${\mathcal E}_1(x), \ldots, {\mathcal E}_n(x)$ measured with the help of $h^*_x$.
Define the mixed $n$-density of $F$-sets by
$$D_n({\mathcal E}_1, \ldots, {\mathcal E}_n) = {\rm V}_{{\mathcal E}_1, \ldots {\mathcal E}_n}\cdot dx,$$
where $dx$ is the Riemannian $n$-density on $X$, and note that this definition does not depend on $h$.
The mixed volume is related to the mixed density by
\begin{equation}\label{m-density}
{\rm V}_n^F({\mathcal E}_1, \ldots, {\mathcal E}_n) = \int_X \, D_n({\mathcal E}_1, \ldots, {\mathcal E}_n).
\end{equation}

\subsection{Products of mixed densities}\label{products}
We want to define a mixed $k$-density $D_k({\mathcal E}_1, \ldots, {\mathcal E}_k)$,
generalizing the above definition of a mixed $n$-density. This will lead us to Theorem {\ref {second}} about the
product of mixed densities, which will be used later in the 
proof of Theorem {\ref {first}}.

For arbitrary tangent vectors $\xi _1, \ldots, \xi_k \in T_x$ let $H\subset T_x$ be the subspace generated by $\xi_i$ 
and let $H^\perp \subset T_x^*$ be the orthogonal complement to $H$.  Given an $F$-set  ${\mathcal E}$ in $X$, denote
by ${\mathcal E}_{\xi_1, \ldots, \xi_k}(x)$ the image of ${\mathcal E}(x)$ under the projection map $T_x^* \to T_x^*/H^\perp$. 
Then $\xi_1\wedge\ldots\wedge \xi_k$ can be considered as a volume form on $T_x^*/H^\perp$, the dual space to $H$.
Put
$$D_k({\mathcal E}) (\xi_1\wedge \ldots \wedge \xi_k) =
\Bigl \vert  \int_{{\mathcal E}_{\xi_1, \ldots, \xi_k}(x)} \xi_1 \wedge \ldots \wedge \xi_k    \, \Bigr \vert.$$
Then, by definition, $D_k({\mathcal E})$ is a $k$-density on $X$.
For a linear combination  of $F$-sets $\sum \lambda _i{\mathcal E}_i$ the expression
$D_k(\lambda _1 {\mathcal E}_1+ \ldots + \lambda _k{\mathcal E}_k)$
is a homogeneous polynomial of degree $k$ in $\lambda _i$. Its coefficient at $\lambda _1\cdot \ldots \cdot \lambda_k$,
divided by $k!$, is denoted by $D_k({\mathcal E}_1, \ldots, {\mathcal E}_k)$.
We call $D_k({\mathcal E})$ the $k$-density of an $F$-set ${\mathcal E}$ and
$D_k({\mathcal E}_1, \ldots, {\mathcal E}_k)$ the mixed $k$-density of ${\mathcal E}_1, \ldots, {\mathcal E}_k$.

If $X$ is equipped with a Riemannian metric then $T_x^*$ is identified with $T_x$, $T_x^*/H^\perp $
with $H$, and ${\mathcal E}_{\xi_1, \ldots, \xi _k}(x)$ with the orthogonal projection
of ${\mathcal E}(x)$ onto $H$. The value $D_k({\mathcal E})(\xi_1
\wedge \ldots \wedge \xi_k)$ is the volume of this projection
multiplied by the length of $\xi_1\wedge \ldots \wedge \xi_k$.

In what follows, 1-densities play a special role. It is easy to see that they can be computed in terms of support functions
as follows.
The function on the tangent bundle, defined by
$$h_{\mathcal E}(x,\xi) = {\rm max}_{\eta \in {\mathcal E}(x)}\, \eta(\xi),$$
is called the support function of ${\mathcal E}$. The corresponding $1$-density is given by
$D_1({\mathcal E})(\xi) = h_{\mathcal E}(x,\xi) + h_{\mathcal E}(x,-\xi)$.

In Secton \ref {ring} we introduce the notion of a normal density on an affine space
 and define the product of normal densities. We also define
the ring ${\mathfrak n}(X)$ of normal densities on $X$ with 
pointwise product, see Theorem \ref {density-manifold}.

Recall that a zonoid is a compact convex body that can be approximated,
in Hausdorff metric, by a Minkowski sum of segments \cite{Sch}. In this paper, we use
the following notion. A {\it smooth} convex body in ${\mathbb R}^n$ is a compact convex set whose support 
function is of class $C^\infty$ on the unit sphere. A smooth convex body is always
$n$-dimensional.

\begin{theorem}\label{second} Let ${\mathcal E}_1, \ldots, {\mathcal E}_k$ be centrally symmetric $F$-sets in $X$. 
Assume that for every $x \in X$ each ${\mathcal E}_i(x)\subset T_x^*X $  is the Minkowski sum of 
a smooth convex body and a zonoid (one of the two summands can be absent).
Then the densities $D_k({\mathcal E}_i)$ are in ${\mathfrak n}(X)$ and
$$D_1({\mathcal E}_1)\cdot \ldots \cdot D_1({\mathcal E}_k) = k!\, D_k({\mathcal E}_1, \ldots, {\mathcal E}_k).$$

\end{theorem}
\noindent
The proof will be given in \ref{densities d_k}.
\begin{corollary}\label{D_pD_q}
Let $p+q = k \le n$. Then
$$D_p({\mathcal E}_1, \ldots, {\mathcal E}_p)\cdot D_q({\mathcal E}_{p+1}, \ldots, {\mathcal E}_k) = {k!\over{p!q!}}
\cdot D_k({\mathcal E}_1, \ldots, {\mathcal E}_k).$$
\end{corollary}
\begin{proof} It suffices to apply Theorem \ref{second} and to write mixed densities as products of $1$-densities.
\end{proof}

\subsection{Hodge inequalities}\label{Hodge}
As a corollary from Theorem \ref{first},
we show here that, for a homogeneous space $X$ of a compact Lie group, the average numbers
of zeros are subject to certain inequalities. We call them Hodge inequalities because they are similar to
the well-known inequalities for intersection indices in algebraic geometry. 

\begin{theorem}\label{H-inequalities}
Let $X$ be a homogeneous space of a compact Lie group. Assume that the vector spaces $V_i$ and their scalar
products $\langle.,.\rangle_i$ are invariant under the given transitive action. Then one has the following Hodge
inequalities:
$$ {\mathfrak M}_X^2(V_1,\ldots,V_{n-1},V_n) \ge {\mathfrak M}_X(V_1,\ldots,V_{n-1},V_{n-1}) \cdot
{\mathfrak M}_X(V_1,\ldots,V_n,V_n) \,. $$
\end{theorem}
\begin{proof}
By Theorem \ref{first} it is enough to prove the inequality
\begin{equation}\label{mv.inequalities}
 {\rm V}_n^F({\mathcal E}_1, \ldots, {\mathcal E}_n)^2 \ge {\rm V}_n^F({\mathcal E}_1, \ldots, {\mathcal E}_{n-1}, {\mathcal E}_{n-1}) \cdot
 {\rm V}_n^F({\mathcal E}_1, \ldots, {\mathcal E}_n, {\mathcal E}_n) .
\end{equation}
In our situation,
the quadratic forms $g_i$ and the ellipsoids ${\mathcal E}_i$ are invariant.
Choose an invariant Riemannian metric on $X$.
Then the mixed density $D_n({\mathcal E}_1, \ldots, {\mathcal E}_n)$ is also invariant
and the mixed volume ${\rm V}_{{\mathcal E}_1, \ldots, {\mathcal E}_n}(x)$ does not depend on $x$.
Furthermore,
$${\rm V}_n^F({\mathcal E}_1, \ldots, {\mathcal E}_n) = {\rm V}_{{\mathcal E}_1, \ldots, {\mathcal E}_n}(x)\cdot {\rm vol}(X)$$ 
for any $x \in X$ by (\ref {m-density}). Fix a point $x \in X$ and let $A_i = {\mathcal E}_i(x)$.
Then (\ref {mv.inequalities}) turns into
$${\rm V}_n^2(A_1, \ldots, A_n)\ge {\rm V}_n(A_1, \ldots, A_{n-1}, A_{n-1})\cdot {\rm V}_n(A_1, \ldots, A_n, A_n),$$
where ${\rm V}_n$
is the mixed volume of compact convex sets. The latter inequalities follow from Alexandrov-Fenchel inequalities, see \cite{Al}.  
\end{proof}
\begin{corollary}\label{H-inequalities*}
In the setting of Theorem \ref{H-inequalities} one has
$${\mathfrak M}_X^n(V_1,\ldots,V_n) \ge {\mathfrak M}_X(V_1)\cdot\ldots\cdot {\mathfrak M}_X(V_n).$$ 
\end{corollary}
\begin{proof} The proof for mixed volumes from \cite{Al} applies.
\end{proof}
Recall that a Riemannian homogeneous space $X = K/L$ is called
isotropy irreducible if the representation of $L$ in the tangent space at the origin is irreducible, see 
\cite{Ga}.
Remark that 
all symmetric spaces of simple compact Lie groups,
e.g., the sphere with the special orthogonal group,
are isotropy  irreducible.
\begin{corollary} \label{isotrop}
If $X$ is isotropy irreducible then we have equalities in Theorem \ref {H-inequalities} and Corollary \ref{H-inequalities*}.
\end{corollary}
\begin{proof} The ellipsoids $A_i$ are balls and their mixed volume is the volume of the unit ball multiplied by the product of 
radii.
\end{proof}

\subsection{Zeros of Laplacian eigenfunctions}\label{Laplace}
In 2003, V.I.Arnold proposed to apply topological invariants to the study of the zero set
of $k\le n$ eigenfunctions of the Laplace operator, see
\cite{Ar1}, Problem 2003--10, p.174. 
He suggested that suitable invariants can be estimated, as in the classical Courant's theorem \cite{CH}, in terms of
the numbers of the corresponding eigenvalues. 

Let $\Delta $ be the Laplace operator on a compact Riemannian manifold $X$ and 
$$H(\lambda) = \{f\in C^\infty(X,{\mathbb R})\ \vert \ \Delta(f) + \lambda f = 0\}$$ 
the eigenspace of $\Delta $
with eigenvalue $\lambda $, considered with $L^2$ metric. Put
$${\mathfrak M}(\lambda _1, \ldots, \lambda _n) = {\mathfrak M}_X(H(\lambda _1), \ldots, H(\lambda _n)),\  \
{\mathfrak M}(\lambda ) = {\mathfrak M}(\lambda,\ldots,\lambda).\ 
$$
If $X$ is a homogeneous space of a compact Lie group and the metric is invariant then
\begin{equation}\label{upper}
{\mathfrak M}(\lambda) \le {2\over \sigma _n n^{n/2}}{\lambda }^{n/2}{\rm vol}\, (X),
\end{equation}
where $\sigma _n$ is the volume of the $n$-dimensional sphere of radius 1, see \cite{AK2}.
In this case we have only one $F$-ellipsoid whose shape with respect to the Riemannian metric does not vary with $x \in X$.
Its semi-axes $\beta _i$ satisfy $\sum \beta _i^2 = \lambda $, see \cite{AK2}. Using this fact, one can
easily deduce (\ref {upper}) from Theorem \ref {first}. Furthermore, 
for isotropy irreducible homogeneous spaces Theorem \ref{first} shows that (\ref{upper}) turns into equality obtained in \cite{AK1},
\cite{Gi}.

The right hand side in (\ref{upper}) coincides, up to a coefficient depending only on $n$,
with the leading term of the asymptotics for the number of $\lambda $ in the celebrated Weyl's law, see \cite{Iv}.
Therefore ({\ref{upper}) can be considered as a step in the direction of Arnold's problem.

We also have the following inequalities of another type.
\begin {theorem} \label{lower} If $X$ is a homogeneous space of a compact Lie group with an invariant Riemannian metric then
$${\mathfrak M}(\lambda_1, \ldots, \lambda _n) ^2 \ge {\mathfrak M}(\lambda _1, \ldots, \lambda_{n-1}, \lambda_{n-1})\cdot{\mathfrak M}(\lambda _1, \ldots, \lambda _n, \lambda _n)$$
and
$${\mathfrak M}(\lambda _1, \ldots, \lambda _n) \ge ({\mathfrak M}(\lambda _1) \cdot \ldots \cdot {\mathfrak M}(\lambda _n))^{1\over n}.$$
\end{theorem}
\begin{proof} The first inequality follows from Theorem \ref{H-inequalities}, the second one from Corollary \ref{H-inequalities*}. 
\end{proof}
\noindent
As an application of our results,
we obtain another proof of the following theorem due to V.M.Gichev, see \cite {Gi}, Thm.\,2, where
one has to take $X=M,\, l = r$ and $t_i = 0$ in (31).
\begin{theorem}\label{Gi}
If $X$ is an isotropy  irreducible homogeneous space of a compact Lie group 
then
$${\mathfrak M}(\lambda _1, \ldots, \lambda _n) = 
{2\over \sigma _n n^{n/2}} \sqrt{\lambda _1 \cdot \ldots \cdot \lambda _n} \, {\rm vol}\, (X).$$
\end{theorem}
\begin{proof} By Corollary \ref{isotrop}   
$${\mathfrak M}(\lambda _1, \ldots, \lambda _n)^n 
=  {\mathfrak M}(\lambda _1)\cdot \ldots \cdot 
{\mathfrak M}(\lambda _n). $$
Also, as we pointed out above, one has
the equality in (\ref{upper}). This completes the proof. 
\end{proof}

\section{Ring of normal densities}\label{ring}
\subsection{Normal densities and normal measures}\label{normal}
Let $V$ be a finite-dimensional real vector space. In considerations involving metric pro\-per\-ties,
we tacitly assume that $V$ has a Euclidean structure
and any vector subspace $U \subset V$ carries the induced metric.  
We will consider translation invariant $k$-densities on $V$. Any of them can be viewed as an even
positively homogeneous function of degree 1 on the cone of decomposable $k$-vectors of $V$.
A density of highest degree coinsides, up to a scalar factor, with the Lebesgue measure on $V$.
Therefore a $k$-density $\delta $ on a vector subspace $U\subset V $ of dimension $ k$
is the Lebesgue measure on $U$ multiplied by some constant $c$.
Since $\delta $ is translation invariant, we can push it to any shift $v_0 + U$.
Given a compact set $B \subset v_0 + U$, we will write $\delta (B)$ for the Lebesgue measure of $B - v_0$
multiplied by $c$.  

We denote by ${\rm Gr}_a(k, V)$ the Grassmanian of affine subspaces of codimension $k$ in $V$
and identify the affine space ${\rm Gr}_a(0,V)$ with the given vector space $V$. The Grassmanian of
vector subspaces of dimension $k$ is denoted by ${\rm Gr}(k,V)$.

\begin{definition}\label{normal measure}
A translation invariant Borel measure on ${\rm Gr}_a(k,V)$, finite on compact sets,
is called a normal measure.
\end{definition}

\noindent {\it Remark.} Normal measures
are Crofton measures as defined in \cite{F}. The pull-back operation for normal measures
introduced below coincides with the corresponding operation for Crofton measures considered by
D.Faifman and T.Wannerer, see Appendix B in \cite{F}.

\medskip
\noindent
For $D\subset V$ we put
\begin{equation}\label{J}
{\mathcal J}_{k,D} = \{H \in {\rm Gr}_a (k, V) \ \vert \ H\cap  D
\ne \emptyset\} 
\end{equation}
Let $\mu _k $ be a normal measure on ${\rm Gr}_a(k,V)$ and
let $\Pi_\xi \subset V$ be a $k$-dimensional 
parallelotope generated by $\xi_1, \ldots, \xi_k\in V$. Define a function on decomposable
$k$-vectors by
$$ \chi _k (\mu _k ) (\xi _1\wedge \ldots \wedge \xi_k) = \mu _k ({\mathcal J}_{k,\Pi_\xi}).$$ 

\begin{proposition} $\chi _k(\mu_k)$ is a translation invariant $k$-density on the affine space $V$.
\end{proposition}

\begin{proof} For a $k$-dimensional vector subspace $U \subset V$
the function $\mu_k({\mathcal J}_{k, D})$ on $k$-dimensional domains $D \subset U$
gives rise to a countably additive and translation invariant measure.
This measure coincides with the Lebesgue measure of $U$ up to a factor
depending continuously on $U$. Thus, for any linear operator $L :U \to U$ one has
$\mu _k({\mathcal J}_{k,L\cdot D}) = \vert{\rm det }(L)\vert \, \mu _k({\mathcal J}_{k,D})$. 
In particular, for a linear operator in the subspace generated by $\xi_i$
and for $D = \Pi_\xi$
we get
$\chi_k(\mu _k)(L\cdot \xi_1 \wedge \ldots \wedge L\cdot\xi_k) =\vert {\rm det }(L)\vert \, \chi_k( \mu _k)
(\xi_1\wedge\ldots\wedge\xi_k)$.
\end{proof}
\noindent
\begin{definition}\label{normal density}
A linear combination of densities of the form $\chi_k(\mu_k)$ is called a {\it normal $k$-density} on $V$.
The space of normal $k$-densities is denoted by ${\mathfrak n}_k$.
\end{definition}
\noindent
Given a linear map $F:U \to V$ of real vector spaces and a normal measure $\mu _k$ on ${\rm Gr}_a(k,V)$
we want to define the pull-back $F^*\mu_k$ on ${\rm Gr}_a(k,U)$.
Let $K = {\rm Ker}\,F$ and let
$${\rm Gr}_a(k,U;K) = \{H \in {\rm Gr}_a(k,U) \ \vert \ H\supset u+K \ \ {\rm for\ some}\ u\in U\}.$$ 
The condition defining ${\rm Gr}_a(k,U;K)$ means that
the vector subspace associated to $H$ contains $K$. Clearly,
${\rm Gr}_a(k,U;K)$ is closed in ${\rm Gr}_a(k,U)$ and translation invariant.
For $T\in {\rm Gr}_a(k,U;K)$
and ${\mathcal T}\subset {\rm Gr}_a(k,U;K)$ put
$$F_*(T) = \{G\in {\rm Gr}_a(k,V) \vert  F(T) = G\cap F(U)\},\ \
F_*({\mathcal T}) = \cup _{T \in {\mathcal T}}F_*(T).$$ 
\begin{proposition}\label{*mu*}
The function ${\mathcal T} \mapsto \mu_k(F_*(\mathcal T)) $ is a normal measure on ${\rm Gr}_a(k,U)$
supported on ${\rm Gr}_a(k,U;K)$.
\end{proposition}
\begin{proof} 
Write $F$ as the composition $$U \ {\buildrel F^\prime \over \to }\ U/K \hookrightarrow V.$$
Then $F^\prime_*(T) = 
\{F^\prime (T)\}$, and so
we get the diffeomorphism 
$${\rm Gr}_a(k,U;K)\ \longrightarrow \ {\rm Gr}_a(k,U/K),\quad  T\mapsto F^\prime_*(T).$$
Thus our statement is reduced to the case of embedding.
We may assume that $U$ is a vector subspace in $V$, $F$ is
the identity map, and $T\in {\rm Gr}_a(k,U)$.
Then $F_*(T) =\{G \in {\rm Gr}_a(k,V)\ \vert \ G\cap U = T\}.$

The closure ${\rm cl}\{F_*(T)\}$  of $F_*(T)$ in ${\rm Gr}_a(k,V)$ consists of all affine subspaces $G$ of codimension $k$,
such that $G\cap U \supset T$. This closure is obviously compact.
Now, if $G_j \in {\rm Gr}_a(k,V),\ T_j \in {\rm Gr}_a(k, U)$ and $G_j\cap U \supset T_j$,
then the convergence of $\{T_j\}$ implies that $\{G_j\}$ has a convergent subsequence. Moreover, if 
$G_j \cap U \ne T_j$ for all $j$, 
then a limit point $G$ of $\{G_j\}$ satisfies ${\rm codim}_U\, G\cap U >k$.
This shows that ${\rm cl}\{F_*(\mathcal T)\}$ and ${\rm cl}\{F_*(\mathcal T)\} -
F_*(\mathcal T)$ are compact if ${\mathcal T}$
is compact. On the other hand, the correspondence ${\mathcal T} \mapsto  F_*({\mathcal T})$,
where ${\mathcal T}$ is any subset of ${\rm Gr}_a(k, U)$, is an 
injective homomorphism of $\sigma$-algebras.
Since the image of a compact set in ${\rm Gr}_a(k, U)$ is the difference 
of two compact sets
in ${\rm Gr}_a(k,V)$,  it follows that the function 
${\mathcal T} \mapsto \mu_k(F_*({\mathcal T}))$ is a correctly defined Borel measure. The translation
invariance is obvious. 
 \end{proof}

\begin{definition}\label{measure pull-back}
The measure defined in Proposition \ref{*mu*} is called the pull-back of $\mu_k$ and is denoted by $F^*\mu_k$.
For ${\mathcal T}\subset {\rm Gr}_a(k,U;K)$ one has
$(F^*\mu_k)({\mathcal T}) = \mu_k(F_*({\mathcal T}))$.

\end{definition}

\noindent Let $F:Y\to Z$ be a differentiable map. Recall that the pull-back of a $k$-density $\nu $ on $Z$
is a $k$-density $F^*\nu $ on $Y$ defined by $F^*\nu (\xi_1\wedge
\ldots \wedge \xi_k) = \nu (dF_p(\xi_1)\wedge \ldots \wedge dF_p(\xi_k))$ for  any $p\in Y $
and $\xi_1, \ldots ,
\xi_k \in T_pY$.

\begin{proposition}\label{n-pull-back}For a linear map $F:U 
\to V$ and a normal measure $\mu _k$ on ${\rm Gr}_a(k,V)$ one has
$\chi_k(F^*\mu_k) = F^*(\chi_k(\mu_k))$.
\end{proposition}

\begin{proof} Let $\xi_1, \dots, \xi _k \in U$ and $\Pi = \Pi_\xi$. Then
$$\chi_k(F^*\mu_k)(\xi_1\wedge \ldots \wedge \xi_k) = F^*\mu_k({\mathcal J}_{k,\Pi})
= \mu_k(F_*({\mathcal J}_{k,\Pi}))
$$
 by Definition \ref{measure pull-back}. On the other hand, 
$$F^*(\chi _k(\mu_k)) (\xi_1\wedge \ldots \wedge \xi_k))= \chi_k(\mu_k)(F\xi_1\wedge\ldots\wedge F\xi_k) = \mu_k({\mathcal J}_{k,F\cdot\Pi})$$
by the definition of pull-back of a density.
It follows from the definition of $F_*$ that $F_*({\mathcal J}_{k, \Pi}) \subset {\mathcal J}_{k,F\cdot \Pi}$.
We have to prove that the value of $\mu_k$ on these two sets is the same.
If ${\rm dim}\,F\cdot \Pi = k$ then the difference
${\mathcal J}_{k, F\cdot \Pi} - F_*({\mathcal J}_{k,\Pi})$ is formed  by subspaces which intersect $F\cdot \Pi$
non-transversally. A small shift of such a subspace  is another subspace, whose intersection with $F\cdot \Pi$
is also non-transversal and disjoint from the initial one. Since $\mu_k$ is translation invariant, it follows from countable additivity 
of $\mu _k$ that 
$\mu_k({\mathcal J}_{k, F\cdot \Pi} - F_*({\mathcal J}_{k,\Pi}) )= 0$.
The same argument shows that 
$$\mu_k (\{G\in {\rm Gr}_a(k,V) \, \vert \, G \supset W\}) = 0$$ for
an affine subspace $W \subset V$ of arbitrary dimension and, finally, that 
 $\mu _k ({\mathcal J}_{k,F.\Pi}) = 0$ if 
${\rm dim}\,F\cdot \Pi < k$.
\end{proof}
Let 
${\mathfrak m}_k$ be the vector space, whose elements are differences of normal measures on ${\rm Gr}_a(k,V)$,
where $k \le n = {\rm dim}\,V$. We will define a product of normal measures
$\mu _p$ on ${\rm Gr}_a(p,V)$ and $\mu _q$ on ${\rm Gr}_a(q,V)$ as a certain normal measure
on ${\rm Gr}_a(p+q, V)$. This 
extends to the product ${\mathfrak m}_p \times {\mathfrak m}_q \to {\mathfrak m}_{p+q}$, where we put $\mu_p\mu _q = 0$
if $p+q >n $,
giving the structure of a ring to the graded space
$${\mathfrak m} = {\mathfrak m}_0\oplus {\mathfrak m}_1 \oplus \ldots \oplus {\mathfrak m}_n.$$

Let $p+q \le n,\ G \in {\rm Gr}_a(p,V), H \in {\rm Gr}_a(q, V)$.
The pair $(G,H)$ is called degenerate if ${\rm codim}\,(G\cap H) \ne p+q $.
The set of degenerate pairs is denoted by ${\mathcal D}_{p,q}$.
For a non-degenerate pair $(G,H)$ write $P_{p,q}(G,H) = G \cap H$.
Thus we have the map
$$P_{p,q}: ({\rm Gr}_a(p,V) \times {\rm Gr}_a(q,V))\setminus {\mathcal D}_{p,q} \ \ \to {\rm Gr}_a(p+q, V).$$
Remark that translations of $V$, acting simultaneously on both factors, leave ${\mathcal D}_{p,q}$ stable
and commute with $P_{p,q}$.
For an affine subspace $H \subset V$ let $d(H)$ be the distance from
the origin to $H$ with respect to some Euclidean metric.
A subset $D \subset {\rm Gr}_a(p,V)$ is said to be bounded if $d(H) < R$ for some $R > 0$ and for all $H\in D$.
Clearly, the notion of boundedness does not depend on the choices of origin and metric. In fact,
$D$ is bounded if and only if $D$ is relatively compact.

\begin{lemma}\label{bounded} If $D \subset {\rm Gr}_a(p+q, V)$ is bounded
then the images of $D$ under the projection mappings of $P_{p,q}^{-1}(D)$ onto ${\rm Gr}_a(p,V)$ and ${\rm Gr}_a(q,V)$
are also bounded.
\end{lemma}

\begin{proof}
Let $L \in D$ and $d(L) <R$. Assume that $L = G\cap H$,
where $G \in {\rm Gr}_a(p, V),\, H \in {\rm Gr}_a (q,V)$. 
Then $d(G)< R$ and $d(H) < R$.
\end{proof}
\noindent
For a bounded domain $D \subset {\rm Gr}_a(p+q, V)$ put
$$(\mu_p\mu_q) (D) = (\mu _p \times \mu_q)(P_{p,q}^{-1}(D))$$
and note that this is finite by Lemma \ref{bounded}.
If $\{D_i\}$ is a decreasing sequence of bounded domains
then the sequence $\{P_{p,q}^{-1}(D_i)\}$ is also decreasing. Moreover, if
$\cap_iU_i = \emptyset$ then $\cap _i P_{p,q}^{-1}(D_i) = \emptyset $.
Since $\mu _p, \mu_q$ are countably additive, $\mu_p\mu_q$ is also
countably additive. Thus $\mu_p\mu_q$ extends to a Borel measure, which is translation invariant
by construction. The resulting normal measure on ${\rm Gr}_a(p+q, V)$ is again denoted by
$\mu_p\mu_q$. 
\begin{lemma}\label{long product} Suppose $p_1+\ldots+p_k\le n$ and let $\mu_i$
be a normal measure on ${\rm Gr}_a(p_i, V), \ i=1, \ldots, n$. Denote by ${\mathcal D}_{p_1, \ldots, p_k}$ 
the subset of ${\rm Gr}_a(p_1, V) \times \ldots \times {\rm Gr}_a(p_k,V)$ formed by all $k$-tuples $(G_1,\ldots,G_k)$,
such that ${\rm codim}\, (G_1\cap\ldots\cap G_k) < p_1+\ldots+p_k$. Then $$(\mu_1\times\ldots\times\mu_k)
({\mathcal D}_{p_1,\ldots,p_k}) = 0.$$
\end{lemma}

\begin{proof} For $k = 1$ the assertion is trivial. Let $k\ge 2 $,
$q = p_1+\ldots+p_{k-1}$.
For $G \in {\rm Gr}_a(q,V)$ put
$${\mathcal D}_G = \{H\in{\rm Gr}_a(p_k, V) \ \vert \ {\rm codim}\, (H\cap G) < p_k + q\}.$$
Assume by induction that $(\mu_1\times\ldots\times\mu_{k-1})({\mathcal D}_{p_1, \ldots, p_{k-1}}) = 0$.
Then, by Fubini's theorem,  it suffices to prove that $\mu_k({\mathcal  D}_G )= 0$.
Now, $\mu_k$ is translation invariant, so we have
$\mu _k ({\mathcal D}_{\epsilon + G}) = \mu_k({\mathcal D}_G)$ for any $\epsilon \in V$. 
Assume first that $q=1$. Then ${\mathcal D}_G$ consists of all those $H \in {\rm Gr}_a(p_k,V)$
which are contained in $G$.
For a generic sequence ${\epsilon _i}$ the non-transversality condition implies ${\mathcal D}_{\epsilon _i +G}
\cap {\mathcal D}_{\epsilon _j +G} = \emptyset $. Choose $\epsilon _i$ so that the series $\sum \epsilon _i$
is normally convergent. Then $\cup\,{\mathcal D}_{\epsilon _i +G}$ is relatively compact.
Since $\mu _k$ is countably additive, we conclude
that $\mu _k ({\mathcal D}_G) = 0$.
Assume now that $q>1$. Choose $\epsilon $ so that $\epsilon +G \ne G$ and put $\epsilon _i = t_i\epsilon$, 
where $t_i\ne t_j$ and $\sum \vert t_i\vert < \infty $. Then ${\mathcal D}_{\epsilon _i + G}\cap 
{\mathcal D}_{\epsilon _j + G} \subset {\mathcal D}_{{\mathbb R}\epsilon + G}$, hence
$\mu_k({\mathcal D}_{\epsilon_i+G }\cap{\mathcal D}_{\epsilon _j +G})=0$  by induction on $q$. 
Also, $\cup_i{\mathcal D}_{\epsilon _i +G}$ is relatively compact.
Applying countable additivity of the measure to the increasing sequence of finite unions
$\cup_{i\le j}{\mathcal D}_{\epsilon_i+G}$,
we obtain our assertion. 
\end{proof}

\begin{corollary} Let $D \subset {\rm Gr}_a(p_1+\ldots+p_k,V)$ be a bounded domain and let
$${\mathcal T}_D =\{(G_1,\ldots,G_k)\in \prod_i{\rm Gr}_a(p_i,V)\ \vert \ D\cap G_1\cap \ldots \cap G_k \ne \emptyset\}.$$
The product of normal measures is associative and
$$(\mu_1\cdot \ldots \cdot \mu_k)(D) = (\mu_1 \times \ldots \times \mu_k)({\mathcal T}_D).$$
\end{corollary}

\begin{lemma}\label{ideal} The subspace $I = {\rm Ker}(\chi_0) \oplus \ldots \oplus {\rm Ker }(\chi_n)$ 
is a homogeneous ideal of the ring $\mathfrak m$.
\end{lemma}
\begin{proof} If $\mu_p \in {\mathfrak m}_p, \chi _p (\mu _p ) = 0$ and  $\mu_q\in{\mathfrak m}_q$
then Fubini's theorem implies $(\mu_p\mu_q)({\mathcal T}_{p+q, \Pi}) = 0$ for any $(p+q)$-dimensional
parallelotope $\Pi$. Therefore ${\chi _{p+q}}(\mu_p \mu_q) = 0$.
\end{proof} \noindent
Let ${\mathfrak n} = {\mathfrak n}_0\oplus {\mathfrak n}_1\oplus \ldots \oplus {\mathfrak n}_n$ denote the graded space of normal densities. Consider the linear map $\chi = \oplus {\chi_i} : {\mathfrak m} \to {\mathfrak n}$ of graded spaces.

\begin{corollary} \label{multiplication}There is a unique structure of a graded ring on ${\mathfrak n}$, such that $\chi $ is a ring homomorphism.
In the notations of Proposition \ref{n-pull-back} the pull-back operations on measures and densities
are ring homomorphisms.
\end{corollary} \noindent
Given a manifold $X$, the graded ring of normal densities $\mathfrak n_x$ is defined for every tangent space $T_x$.
A density $\delta $ on $X$ is called normal if $\delta _x \in {\mathfrak n}_x$ for every $x \in X$.
The set of normal densities on $X$ is denoted
by ${\mathfrak n}(X)$. With respect to pointwise multiplication, ${\mathfrak n}(X)$ is a commutative graded algebra
over $C(X)$.  
\begin{theorem}\label{density-manifold} For any differentiable map $F: X \to Y$ 
the pull-back operation $\delta \mapsto F^*\delta$,
$(F^*\delta)_x = (dF_x)^*\delta _{F(x)}$, 
defines a homomorphism ${\mathfrak n}(Y) \to {\mathfrak n}(X)$. 
The assignment $X \to {\mathfrak n}(X)$ is a contravariant functor from the category
of differentiable manifolds to the category of commutative graded rings. For any $\delta \in 
{\mathfrak n}(Y)$ and $f\in C(Y)$ one has $F^*(f\delta) = (f\circ F)\cdot F^*\delta $.

\end{theorem}
\begin{proof} 
The pull-back of a normal density is normal by Proposition \ref{n-pull-back}. 
The fact that above map ${\mathfrak n}(Y) \to {\mathfrak n}(X)$ is a ring homomorphism follows from Corollary \ref{multiplication}.
\end{proof}

\subsection{Normal densities and the cosine transform.}\label{cos}
For a Euclidean structure on $V$ let ${\rm vol}_k$ denote
the corresponding Riemannian $k$-density, i.e., the translation invariant $k$-density, whose value
on the $k$-vector $\xi _1 \wedge \ldots \wedge \xi_k$ equals the $k$-dimensional volume of the
parallelotope generated by $\xi_1, \ldots, \xi_k$. Starting with a continuous real function $\phi $
on ${\rm Gr}(k,V)$, we will define a translation invariant $k$-density $\delta _{k,\phi}$ on $V$
and a normal measure ${\mu}_{k,\phi} \in {\mathfrak m}_k$. The function $\phi $ will be called the gauge function
of $\delta _{k,\phi}$ and $\mu _{k,\phi}$.

The density $\delta _{k,\phi}$ is defined by
\begin{equation}\label{k-density}
\delta_{k,\phi}(\xi_1 \wedge \ldots \wedge \xi_k) = \phi(H)\cdot {\rm vol}_k(\xi _1 \wedge \ldots \wedge \xi _k),
\end{equation}
where $H$ is the subspace generated by linearly independent vectors 
$\xi_1, \ldots \xi_k$. Note that any translation invariant $k$-density can be written in this form.
A translation invariant $k$-density is said to be of class $C^\infty$ if the associated gauge function
$\phi : {\rm Gr}(k,V) \to {\mathbb R}$ is of class $C^\infty$. 

In the special case of normal densities, we attach the gauge function to a density
and identify ${\mathfrak n}_k$ with a (non-closed) vector 
subspace of $C({\rm Gr}(k,V))$ considered with
the topology of uniform convergence.
We will sometimes indicate the ambient space and write ${\mathcal J}_{p,C,V}$ 
in place of ${\mathcal J}_{p,C}$ (see (\ref {J})).

\begin{lemma}\label{estimate} Assume that the normal measure $\mu _p \in {\mathfrak m}_p$ is non-negative.
If $C \subset V$ is a cube in some coordinate system then
$$\mu _p({\mathcal J}_{p,C} )\le a(C)\cdot m,$$
where $m$ is the maximum of the gauge function of $\chi _p(\mu_p)$ and the constant $a(C)$ does not depend on $\mu _p$.
\end{lemma}

\begin{proof} If $p = n ={\rm dim}\, V$ then $\mu _p$ is a Lebesgue measure and the estimate holds true with $a(C)$
being the volume of $C$. In particular, the assertion of the lemma is obvious if $n = 1$. Suppose $n > 1$. Let $W_1, 
\ldots, W_n \subset V$ be the coordinate subspaces of codimension 1 
 and let $C_1,\ldots,C_{2n}$ be the $(n-1)$-dimensional faces of $C$. 
We can assume that the numbering is chosen so that $C_i \subset W_i$
for all $i = 1,\ldots, n$.
Clearly,
$${\mathcal J}_{p,C,V}= \bigcup _{i=1}^{2n}\, {\mathcal J}_{p,C_i,V}.$$
On the other hand, denote by $\mu _p^{(i)}$ the pull-back of $\mu _p$
under the embedding $W_i \to V$. By Definition \ref{measure pull-back} we have
$$\mu_p({\mathcal J}_{p,C_i, V}) = \mu _p^{(i)} ({\mathcal J}_{p,C_i,W_i}),\ i =1, \ldots, n.$$
By Proposition \ref{n-pull-back} the gauge functions of $\chi_p(\mu _p^{(i)})$ do not exceed $m$.
Thus
$$\mu_p({\mathcal J}_{p,C,V}) \le \sum _{i=1}^{2n}\mu _p({\mathcal J}_{p,C_i,V}) = 2\sum_{i=1}^n \mu_p^{(i)}
({\mathcal J}_{p,C_i,W_i})\le 2n \cdot {\rm max}\, a(C_i)\cdot m$$
by induction.
\end{proof}
The following property of the product
${\mathfrak n}_p \times 
{\mathfrak n}_q \to {\mathfrak n}_{p+q}$ will be useful.
 Let $\delta  = \chi_p(\mu_p)$ and $\delta ^\prime = \chi_q(\mu _q)$  be a normal $p$-density
and a normal $q$-density, respectively.

\begin{proposition}\label{continuous} Suppose  $\{\delta _i\}
\subset {\mathfrak n}_p$ and $\{\delta _i^\prime\} \subset {\mathfrak n}_q$
are two sequences of normal densities converging to $\delta $ and
$\delta ^\prime$, respectively. Assume that
$\delta _i^\prime = \chi_q(\mu _{q,i})$, where all normal measures $\mu _{q,i}$ are non-negative.
Then $\delta _i\delta _i^\prime $ tends to $\delta \delta ^\prime$.

\end{proposition}
\begin{proof}
 Let $A$ be a compact convex $(p+q)$-dimensional set in $V$, $G$
an arbitrary point in ${\rm Gr}_a(q,V)$ and $A_G = A\cap G$. Then
$$(\delta \delta ^\prime)(A) = \int _{G\in {\rm Gr}_a(q,V)} \delta (A_G) d\mu _q(G)$$
by the definition of the product of normal densities and by Fubini's theorem. Suppose now that $\{\delta _i\}
\subset {\mathfrak n}_p$ and $\{\delta _i^\prime\} \subset {\mathfrak n}_q$
are two sequences of normal densities converging to $\delta $ and
$\delta ^\prime$, respectively. Then $\delta _i(A_G) $ tends to
$\delta (A_G)$ uniformly in $G$ with $A_G$ non-empty, hence
$\delta _i \delta^\prime \to \delta \delta ^\prime$ and, similarly, $\delta \delta _i^\prime \to \delta \delta ^\prime $. Finally,
write $\delta _i 
\delta _i^\prime =  (\delta _i - \delta) \delta _i ^\prime +\delta \delta_i^\prime$. We have to show that the first summand
tends to 0.
Since the sequence of gauge functions of $\chi _q(\mu _{q,i})$ is convergent, Lemma \ref{estimate}
shows that the measures $\mu _{q,i}(\{G\in {\rm Gr}_a(q,V) \vert A_G\ne \emptyset\})$ are bounded
by the same constant. This completes the proof.
\end{proof}

An affine subspace $G \in {\rm Gr}_a(k,V)$ has a unique presentation as the sum $G = h+H^\bot $, where $H \in {\rm Gr}(k,V)$,
$H^\bot $ is the orthogonal complement to $H$ and $h\in H$. The metric on $V$ induces a metric on $H$
and the associated Lebesgue measure on $H$ is denoted by $dh$. The measure $\mu_{k,\phi}$
is given by integration against compactly supported functions. Namely, if $\psi $ is such  function on ${\rm Gr}_a(k,V)$ then
\begin{equation}\label{affine-to-vector}
\int_{{\rm Gr}_a(k,V)}\psi \cdot d\mu _{k,\phi} = \int_{{\rm Gr}(k,V)}\phi(H)\,
 \Bigl (\, \int _H \psi(h+H^\bot)\cdot dh \Bigr )
\cdot dH ,
\end{equation}
where $dH$ is the Haar measure on the Grassmanian.

Recall the definition of the cosine transform. Given $E \in {\rm Gr}(k,V) ,F\in {\rm Gr}(l,V),$ where $k\le l$ ,
let $A \subset E$ be any subset of non-zero volume. The
cosine of the angle between $E$ and $F$ is the ratio of the $k$-dimensional volume of 
the orthogonal projecion of $A $ onto $F$ to
the $k$-dimensional volume of $A$. The ratio is denoted here by ${\rm cos}(E,F)$ (some authors
write $\vert {\rm cos}(E,F)\vert$ in a more classical way). The cosine transform $T_k: C({\rm Gr}(k,V))
\to C({\rm Gr}(k,V))$ is the integral operator
$$T_k(f)(G) = \int_{{\rm Gr}(k,V)}f(H)\,{\rm cos}(H,G)\, dH.$$

\begin{proposition}\label{cos-transf} Let $\Phi = T_k(\phi)$. Then $\delta _{k,\Phi} = 
\chi_k(\mu _{k,\phi})$, i.e., the functions from the image of $T_k$ are gauge
functions of normal densities and the diagram
$$\begin{matrix}
&C({\rm Gr}(k,V)) &\buildrel T_k \over \longrightarrow & {\rm Im}\,T_k & \subset \ C({\rm Gr}(k,V))\cr
&\downarrow &&\downarrow  \cr
&{\mathfrak m}_k &\buildrel {\chi_k}\over \longrightarrow  &  \ {\mathfrak n}_k  &  \cr$$
\end{matrix}$$
commutes, 
where the mappings denoted by vertical arrows attach the measure and, respectively,  the density to a gauge function.
\end{proposition}
\begin{proof} Let $\psi $ denote the characteristic function of ${\mathcal T}_{k,\Pi_\xi}$
and let $G_\xi $ be the vector subspace generated by $\xi_1,\ldots ,\xi _k$.
Then
$$\chi_k(\mu_{k,\phi})(\Pi_\xi) =\int _{G\in {\rm Gr}(k,V)} \phi (G) \Bigl(\int_G \psi (g+G^\bot) dg\Bigr) dG=
$$

$$= \int_{{\rm Gr}(k,V)} \phi (G) {\rm vol}_k(\Pi_\xi)\, {\rm cos}(G, G_\xi) dG = \Phi (G_\xi){\rm vol}_k(\Pi_\xi)
= \delta _{k,\phi}(\Pi _\xi). $$
\end{proof}
One can mimic the definition (\ref {affine-to-vector}) of $\mu _{k,\phi}$  replacing the function $\phi $ by a Borel measure $\nu $
on ${\rm Gr}(k,V)$. Namely, define the normal measure $\mu _{k,\nu}$ by
$$\int _{{\rm Gr}_a(k,V)}
\psi\cdot d\mu _{k,\nu} = 
\int _{{\rm Gr}(k,V)}\Bigl (\int _H \psi (h+H^\bot)\cdot dh\Bigr)\cdot d\nu(H).
$$
\begin{proposition}\label{cos-transf*}
Let
$$T_k(\nu) (G) = \int _{{\rm Gr}(k,V)} {\rm cos}(H,G)\cdot d \nu (H). $$
If $\Phi = T_k(\nu )$ then
$\delta _{k,\Phi} = \chi_k(\mu_{k,\nu})$.
\end{proposition}
\begin{proof} It suffices to replace $\phi (G)dG$ by $d\nu(G)$ in the proof of Proposition \ref{cos-transf}.
\end{proof}
\begin{proposition}\label{1-normal} Every $C^\infty$ translation invariant 1-density is normal.
\end{proposition}
\begin{proof} By Proposition \ref{cos-transf} a density of the form $\delta _{1,T_1(\phi)}$ is normal. 
On the other hand, the restriction of the cosine transform to $C^\infty $ functions
 is an automorphism $$T_1: C^\infty ({\rm Gr}(1,V)) \to C^\infty({\rm Gr}(1,V)).$$ 
In particuar, 
any $C^\infty $ function on ${\rm Gr}(1,V)$
is in the image of the cosine transform, see \cite{AB}.
\end{proof}
We will need the notion of push-forward for densities. Let $F:Y \to Z$ be a fibration with $l$-dimensional
fiber $Y_z,\ z\in Z,$ and let $\mu $ be a $k$-density on $Y$. 
Given $e_1\wedge\ldots\wedge e_{k-l} \in \bigwedge ^{k-l} T_zZ$, choose tangent vectors $h_1, \ldots, h_{k-l} \in T_yY$,
such that $dF_y(h_i) = e_i$. Then the $l$-density on the fiber $F_z$, given by
$$i_{h_1\wedge \ldots \wedge h_{k-l}}\mu :  v_1\wedge \ldots \wedge v_l \mapsto \mu (h_1\wedge \ldots \wedge h_{k-l}\wedge v_1 \wedge \ldots \wedge v_l),$$
is independent of the choice of $h_i$. The
push-forward of $\mu $ is defined as the $(k-l)$-density
$$F_*
\mu (e_1\wedge \ldots \wedge e_{k-l}) = 
\int _{F_z}\, i_{h_1
\wedge\ldots\wedge h_{k-l}} \nu,$$
provided the integral is finite.

Let $U \subset V$ be a vector subspace of dimension $k$. 
Define the mapping $$\nu : {\rm Gr}(1,V) \setminus {\rm Gr}(1, U^\bot) \to {\rm Gr}(1,U)$$ by
$\nu (H) = (H+U^\bot)\cap U$. The $\nu $ is a fibration with the fiber
$$\nu^{-1}(L) = {\rm Gr}(1,L+U^\bot) \setminus {\rm Gr}(1,U^\bot).$$
Let $dH$ be the normalized density of highest degree on ${\rm Gr}(1,V)$, i.e., the density
of the Haar measure on the projective space. The push-forward of a continuous real function
$g $ on ${\rm Gr}(1,V) $ is defined by the equality
$$\nu_*(g\cdot dH) = (\nu_*g)\cdot dL,$$
where $dL$ is the normalized density of highest degree on ${\rm Gr}(1,U)$. Note that if $\phi : {\rm Gr}(1,V)
\to {\mathbb R}$
is constant on the fibers of
$\nu $ then $\nu_*(\phi g) = \phi \cdot \nu_*g$.

For future use we prove the following proposition.
\begin{proposition} \label{T_1-preimage}
Let $f\in C^\infty ({\rm Gr}(1,V))),\ f_U$ the restriction of $f $ to ${\rm Gr}(1,U)$,
and $g(H) = T_1^{-1}f(H)\cdot{\rm cos }(H,U).$ Then
$$T_{1,U}^{-1}f_U = \nu _*g,$$
where $T_{1,U}$ is the cosine transform on ${\rm Gr}(1,U)$.
\end{proposition}
\begin{proof} For $L \in {\rm Gr}(1,U)$ one has ${\rm cos}(H,L) = {\rm cos}(H,U)\,{\rm cos}(\nu(H),L)$,
hence
$$f_U(L) = \int_{{\rm Gr}(1,V)}\{T_1^{-1}f(H)\}{\rm cos}(H,L) \,dH = \int_{{\rm Gr}(1,V)}g(H){\rm cos}(\nu(H),L)dH.$$
From the defintion of $\nu_*g$, it follows that
$$f_U(L) = \int_{{\rm Gr}(1,U)}(\nu_*g)(K){\rm cos}(K,L)dK = T_{1,U}(\nu _*g)(L).$$
\end{proof}

\subsection{Normal densities and valuations of convex bodies}\label{val}
This small subsection contains a number of remarks that are neither proved, nor used in the rest of the paper.
Let $v$ be a function on the set of compact convex bodies in ${\mathbb R}^n$.
Then $v $ is called a valuation if 
$v (A\cap B) + v (A\cup B) = v(A) + v (B)$
for any two compact convex bodies $A,B$, such that $A\cup B$ is also convex.
For the theory of valuations the reader is referred to \cite{A} and references therein.
We assume that $v $ is continuous in Hausdorff metric, translation invariant and even, i.e., $v (-A) = v(A)$. 
The valuation $v $ is called $k$-homogeneous if $v (tA) = \vert t \vert^k v(A)$.
A normal $k$-density $\delta $, as a function on $k$-dimensional parallelotopes, extends to a $k$-homogeneous
valuation $v_\delta $. Indeed, 
following Definition \ref{normal density}, we may assume that $\delta = \chi _k(\mu _k)$, where
$\mu _k$ is a normal measure. Then we put $v_\delta (A) = \mu _k({\mathcal J}_{k,A})$,
where ${\mathcal J}_{k,A}$ is defined by (\ref{J}). In other words, normal densities are contained
in the image of Klain map, see \cite{Kl}.

{\it Remark 1}. If a normal $k$-density $\delta $ is smooth (of class $C^\infty$) then the valuation $v_\delta$ is also smooth
in the sense of \cite{A}.
The assignment $\delta \mapsto v_\delta$ defines a one-to-one correspondence
between the sets of smooth normal $k$-densities and smooth $k$-homogeneous valuations.

{\it Remark 2}. The product of smooth normal densities
agrees with Alesker product of smooth valuations. Therefore the equality
$$d_p(A_1, \ldots, A_p)\cdot d_q(A_{p+1}, \ldots, A_{p+q}) = {(p+q)!\over p!q!}d_{p+q}(A_1, \ldots, A_{p+q}),$$
following from Theorem \ref{A_1,...,A_k} in Section \ref{proofs},   
can be regarded as a computation of the product
of valuations related to smooth centrally symmetric bodies $A_i$. 

{\it Remark 3.} The product of measures on affine Grassmanians and its connection with Alesker product
of valuations is considered in \cite{Bg}.
 
{\it Remark 4}. The pull-back operation on valuations of convex bodies is defined in \cite{A1}. This operation
agrees with our pull-back operation for normal measures.
 
\section{Proofs of main resuts}\label{proofs}
\subsection{Some facts from convex geometry}\label{convex}
Let $V={\mathbb R}^n$ be a Euclidean space, $S^{n-1} \subset V$ the unit sphere. 
For a vector subspace $M \subset V$ we denote by $M^\perp $ its orthogonal complement and by $\pi _M$
the projection map $V \to M$.
For $x \in S^{n-1}$ we write $x^\perp$ instead of $({\mathbb R}\cdot x)^\perp$.

Let
 $A \subset V$  be a compact convex set of dimension $k$, 
${\rm V}_k(A)$  its $k$-dimensional volume and $h_A$ the support function. Later on, it will be also convenient
to define the width function of $A$ on the projective space ${\rm Gr}(1,V)$ by
$$s_A(H) = h_A(x) + h_A(-x) = {\rm V}_1(\pi _HA),$$
where $H \in {\rm Gr}(1,V)$ and $x$ is a unit vector in $H$.
The cosine transform
on the unit sphere $$f \mapsto (Tf)(x) = \int _{S^{n-1}}f(s)\vert (x,s) \vert ds$$
will be considered as a linear operator on even functions which are identified
with functions on ${\rm Gr}(1,V)$. 
Then $T$ is invertible on the space
of even $C^\infty $ functions, see \cite{AB}. Furthermore, if $A$ is centrally symmetric with center 0  
and $h_A$ is smooth, then $h_A$ is contained in the image of $T$.
The proof of the following result is due to S.Alesker.
\begin{lemma}\label{Alesker} For a smooth centrally symmetric body $A$ with center 0 one has
$$\int _{S^{n-1}}T^{-1}h_A(x)\, {\rm V}_{n-1}(\pi _{x^\perp}A)dx = {n\over 2}\,{\rm V}_n(A).$$
\end {lemma}
\begin{proof}
We may assume that the Gaussian curvature $K$ of $\partial A$ does not vanish. Indeed, one can approximate $A$ by a convex cenrally symmetric smooth body having this property and
then use the continuity of $T^{-1}$ in $C^\infty $-topology.
Recall that ${\rm V}_n(A) = {1\over n} \int_{S^{n-1}}h_A(x)K(x)^{-1}dx$.
Since $T$ is a self-adjoint operator in $L^2(S^{n-1}, dx)$, we have
$$n{\rm V}_n(A) = \int _{S^{n-1}}h_A(x)\, K^{-1}(x)dx = \int _{S^{n-1}}T^{-1}h_A(x)\, TK^{-1}(x)dx=$$
$$= 2\int_{S^{n-1}}T^{-1}h_A(x)\, {\rm V}_{n-1}(\pi _{x^\perp}A)dx,$$ 
where we used the identity $TK^{-1}(x) = 2{\rm V}_{n-1}(\pi _{x^\perp}A)$.
\end{proof}
\noindent
We now want to restate the assertion of Lemma \ref{Alesker}
in terms of the Haar measure $dH$ on the projective space ${\rm Gr}(1,V)$. Note that for $A$
symmetric $s_A(H) = 2h_A(x)$, where $x$ is a unit vector in $H$.

\begin{corollary}\label {Haar-1}
Let $T_1$ be the cosine transform on $Gr(1,V)$. Under the above assumptions
$$\int_{{\rm Gr}(1,V)}T_1^{-1}s_A(H)\, {\rm V}_{n-1}(\pi _{H^\perp}A)dH = n{\rm V}_n(A).$$
\end{corollary}
\smallskip
\noindent
The next proposition deals with a vector subspace $D \subset V$ . For $D = V$ we retrieve 
Corollary \ref{Haar-1}. 
\begin{proposition}\label{Haar-2} Let $D\subset V$
be a $k$-dimensional vector subspace. Then under the above asumptions
$$\int_{{\rm Gr}(1,V)} T_1^{-1}s_A(H)\,{\rm cos}(H,D){\rm V}_{k-1}(\pi _{H^\perp\cap D}A)dH = k{\rm V}_k(\pi_DA).$$  
\end{proposition}
\begin{proof}
For $H \in {\rm Gr}(1,V)\setminus{\rm Gr}(1,D^\perp)$ put 
$\nu (H) = (H+D^\perp)\cap D$. By Proposition \ref{T_1-preimage} we obtain
$T^{-1}_{1,D}(s_A)_ D = \nu_*g$,
where 
$$g(H) = T_1^{-1}s_A(H)\,{\rm cos}(H,D).$$
Let $\phi (H) = {\rm V}_{k-1}(\pi_{H^\perp\cap D}A) = {\rm V}_{k-1}(\pi _{L_H^\perp}\pi _DA)$, where $L_H^\perp$
is the ortho\-gonal complement to $L_H =\nu (H)$ in $D$. Since $\phi $ is constant along the fibers of $\nu $,
we have
$\nu _*(\phi g) = \phi \nu_*(g)$.
It follows that the integral on the left hand side equals
$$\int _{{\rm Gr}(1,D)}\bigl (T_{1,D}^{-1}(s_A)_D\bigr )(L)\, {\rm V}_{k-1}(\pi _{L^\perp}\pi _DA)\, dL.$$
Applying Corollary \ref{Haar-1} to the convex body $\pi _DA \subset D$, we get the desired equality.
\end{proof}

\subsection{Densities {$\bf d_k(A_1,\ldots,A_k)$}}\label{densities d_k}
Suppose we are given $k$ compact convex sets $B_1, \ldots, B_k$ in a $k$-dimensional vector subspace of
$V = {\mathbb R}^n$. Then their mixed volume is denoted by 
${\rm V}_k(B_1, \ldots, B_k)$. For any compact convex sets $A_1, \ldots, A_k \subset V$ we have the associated translation invariant
$k$-density $d_k(A_1, \ldots, A_k)$, defined as follows. Let $\xi_1, \ldots, \xi_k \in V$ and let $H$ be the vector subspace
generated by $\xi_1, \ldots, \xi_k$. Then
$$d_k(A_1, \ldots, A_k)(\xi_1\wedge\ldots\wedge \xi_k) =
{\rm V}_k(\pi _HA_1, \ldots, \pi_HA_k)  \cdot {\rm vol}_k(\xi_1\wedge \ldots \wedge \xi_k)$$
if ${\rm dim}\, H = k$ and
$d_k(A_1, \ldots, A_k) = 0$ if ${\rm dim}\, H < k$.
We also use the notation $d_k(A) = d_k(A, \ldots, A)$, where $A$ appears $k$ times on the right hand side.

Recall that a translation invariant $k$-density can be evaluated on compact subsets contained in a shift of a $k$-dimensional
vector subspace, see \ref{normal}. Recall also that normal densities can be multiplied, see Corollary \ref{multiplication}.

\begin{proposition}\label{d_1 normal} Let $A$ be a smooth convex body.
Then the density $d_1(A)$ is normal.
\end{proposition}
\begin{proof} Let $\xi \in V, \xi\ne 0$, and $H = {\mathbb R}\xi$. 
Then $$d_1(A)(\xi) = {\rm V}_1(\pi_HA)\cdot {\rm vol }_1(\xi) = s_A(H)\cdot {\rm vol}_1(\xi )
= \delta_{1,s_A}(\xi)  $$
by (\ref{k-density}). 
Since $s_A$ is a smooth function, we can put
$\phi = T_1^{-1}(s_A)$. Then
$$d_1(A) = \delta_{1,s_A}= \chi_1(\mu _{1, \phi})$$
by Proposition \ref{cos-transf}, showing that $d_1(A)$ is normal.
\end{proof}

\begin{proposition} \label{d_1,d_k} Let $A$ be as in Proposition \ref{d_1 normal} and centrally symmetric.
If $B \subset V$ is a compact convex set of dimension $k$, contained in a shift 
of a $k$-dimensional
vector subspace $D \subset V$, then
$$(d_1(A))^k(B) = k!\,d_k(A)(B).$$
\end{proposition}

\begin{proof}
We will prove our statement by induction.
For $k = 1$ there is nothing to prove, so let $k > 1$. By the definition of the product of normal densities
$\delta_p = \chi_p(\mu _p), \delta_q = \chi_q(\mu_q)$
and by Fubini's theorem
$$\delta _p\delta _q (B) =(\mu _p \times \mu_q)\{(I,J)\in {\rm Gr}_a(p,V)\times{\rm Gr}_a(q,V) 
\vert I \cap J \cap B \ne \emptyset\}=$$
$$=\int _{I \in {\rm Gr}_a(p,V)} \mu_q\{J\in {\rm Gr}_a(q,V )\,\vert \, I\cap J \cap B \ne \emptyset \} d\mu_p,$$
where $B$ is a compact convex set of dimension $p+q$. Take $p=1, \, q= k-1, \,
\delta _1= d_1(A), \, \delta _{k-1} = d_1(A)^{k-1}$ and apply the induction hypothesis.
We get
$$(d_1(A))^k(B)
= \int_{I\in {\rm Gr}_a(1,V)} d_1(A)^{k-1}(B\cap I) \, d\mu_{1,\phi}=$$
$$=(k-1)! \int_{I\in {\rm Gr}_a(1,V)}{\rm V}_{k-1}(B\cap I)\, {\rm V}_{k-1} (\pi _{D_I}A)\,d\mu_{1,\phi},$$
where $D_I$ is the intersection of $D$ with the hyperplane through the origin parallel to $I$.
Using (\ref {affine-to-vector}) rewrite this as
$$(d_1(A))^k(B) =$$
$$ =(k-1)! \int_{H\in {\rm Gr}(1,V)} \phi(H) {\rm V}_{k-1}(\pi _{H^\perp \cap D}A) \Bigl( \int_HV_{k-1}((h+H^
\perp)\cap B)dh \Bigr ) dH.$$
Now notice that 
$$\int _H{\rm V}_{k-1}((h+H^\perp)\cap B) dh = {\rm V}_k(B)\, {\rm cos}(H,D),$$
hence
$$(d_1(A))^k(B) =$$
$$= (k-1)! {\rm V}_k(B)\int_{H\in {\rm Gr}(1,V)}T_1^{-1}s_A(H){\rm V}_{k-1}
(\pi _{H^\perp\cap D}A){\rm cos}(H,D)dH.$$
A parallel shift of $A$ does not change the density $d_1(A)$.
Therefore we may assume that the center of symmetry of $A$ is 0.
Then Proposition \ref{Haar-2} applies and our assertion follows.
\end{proof}

By definition, a zonotope is a Minkowski sum of finitely many segments and a
zonoid is a convex body that can be approximated, in Hausdorff metric, by a sequence
of zonotopes. Zonoids are known to have a center of symmetry. The zonoids with center 0 are characterized  by the fact
that their support functions are obtained via cosine transform from non-negative even measures on the sphere, see \cite{Sch}.

\begin{lemma}\label{zonoid} A zonoid can be approximated by a sequence of smooth zonoids.
\end{lemma}

\begin{proof} The support function of the Minkowski sum of finitely many convex bodies is the sum of support functions
of the summands. Thus, if all bodies are smooth then
their Minkowski sum is also smooth. Therefore it suffices to prove the lemma
for a segment. Consider any plane containing the segment $[a,b]$ and take the ellipse with focuses $a,b$ in that plane.
The segment can be approximated by ellipsoids obtained by rotating the ellipse
around the axis through $a$ and $b$. It remains to show that any ellipsoid is a zonoid. This is clear for the sphere
with center 0 because its support function is the cosine transform of a constant function. Since zonoids form
an affine-invariant class of convex bodies, all ellipsoids are zonoids.
\end{proof}

\begin{lemma}\label{zonoid1} If a compact convex body $A$ is a zonoid then the density $d_1(A)$ is normal.
\end{lemma}
\begin{proof} We may assume that $A$ has center 0. Let $\nu $ be the even measure
on the sphere, i.e., the measure on ${\rm Gr}(1,V)$, such that $T_1(\nu ) = s_A$. Then
$$d_1(A) = \delta_{1,s_A} = \chi_1(\mu _{1,\nu})$$
by Proposition \ref{cos-transf*}.
\end{proof}
\begin{theorem}\label{A_1,...,A_k}
Let $A_1,\ldots,A_k$ be centrally symmetric compact convex sets in $V$. Assume that each $A_i$ is either smooth, or is a zonoid, or else is the Minkowski sum of the bodies of those two types.
Then
$$d_1(A_1)\cdot \ldots \cdot d_1(A_k) = k!\,d_k(A_1, \ldots, A_k).$$
\end{theorem}
\begin{proof} Note that $d_1(A+B) = d_1(A) + d_1(B)$, so all densities $d_1(A_i)$ are normal
by Proposition \ref{d_1 normal} and Lemma \ref{zonoid1}. Thus
the multiplication of $d_1(A_i)$ makes sense.
Assume first that all $A_i$ are smooth convex bodies.
The expression $(d_1(\lambda _1A_1 + \ldots + \lambda _kA_k))^k$
is a homogeneous polynomial of degree $k$ in $\lambda _i$. The polarization formula for this polynomial
reduces 
the theorem to Proposition \ref{d_1,d_k}. 

More generally, assume that each convex body is of the form $A_i +Z_i$, where $Z_i$ is a zonoid. 
 Using Lemma \ref{zonoid}, we can
approximate $Z_i$ by smooth zonoids $Z_{ij}$.
Then
$$d_1(A_1+Z_{1j})\cdot \ldots \cdot d_1(A_k+Z_{kj}) = k!\,d_k(A_1+Z_{1j}, \ldots, A_k+Z_{kj})$$
for any $j$. 
As $j \to \infty$, we have $d_1(A_i+Z_{ij}) \to d_1(A_i +Z_i)$
and $d_k(A_1+Z_{1j},\ldots,A_k+Z_{kj}) \to d_k(A_1+Z_1, \ldots, A_k+Z_k)$.
Moreover, the measures defining the densities $d_1(Z_{ij})$ are non-negative. Thus 
the assumptions of Proposition \ref{continuous} are fulfilled for all sequences $\{d_1(Z_{ij})\}$, where $ i=1,\ldots,k$. It follows 
that the product on the left hand side tends to $d_1(A_1+Z_1)\cdot \ldots \cdot d_1(A_k + Z_k)$. 
\end{proof}
\medskip
\noindent
{\it Proof of Theorem \ref{second}.} Using an arbitrary Euclidean metric in the tangent space $T_xX$,
identify $T_xX$ with its dual $T^*_xX$. Let $H$ be the subspace generated
by $\xi _1, \ldots, \xi_k \in T_xX$. The definition of $D_k({\mathcal E})$ in \ref{products} implies
$$D_k({\mathcal E})(\xi_1 \wedge \ldots \wedge \xi_k) = {\rm V}_k(\pi _H{\mathcal E}(x))\cdot
 {\rm vol}_k(\xi _1 \wedge \ldots \wedge \xi_k).$$
This shows that the density $D_k({\mathcal E})$ and the mixed density
$D_k({\mathcal E}_1, \ldots, {\mathcal E}_k)$ on $T_xX$ coincide with $d_k({\mathcal E}(x))$ and
$d_k({\mathcal E}_1(x), \ldots, {\mathcal E}_k(x))$, respectively. Therefore the required assertion
follows from Theorem  \ref{A_1,...,A_k}.\hfill{$\square $}

\subsection{Crofton formula}\label{Crofton}
 For $i=1,\ldots,n$, let $E_i$ be a finite-dimensional
real vector space with scalar product $\langle.,.\rangle_i$ and let $E_i^*$ be the dual Euclidean space. We denote by $S_i$ and $S_i^*$ the unit spheres in
$E_i$ and $E_i^*$, respectively. Put
$E=E_1\times\ldots\times E_n, \ E^*=E_1^*\times \ldots \times E_n^*,\ S=S_1\times\ldots\times S_n$, and $S^*= S_1^*\times\ldots\times S_n^*$.  We consider a point $s_i^* \in S_i^*$ as a linear function on $E_i$ and,
also, as a linear function on $E$ obtained by an obvious lifting using the projection $ E \to E_i$.
Thus, the $n$-tuple $s^* =(s_1^*,\ldots,s_n^*) \in S^*$ is a system of functions on $E$.
Let $ds_i$ and $ds_i^*$ be the Euclidean volume densities on $S_i $ and $S_i^*$, respectively.
We denote by $ds$ and $ds^*$ their normalized products, i.e., the densities on $S$ and $S^*$
equal to
$$ds = {1\over {\sigma }}\prod _i ds_i,\ ds^* = {1\over {\sigma }}\prod _i ds_i^*,$$
where ${\sigma } $ is the product of volumes of the unit spheres $S_i$ (or $S_i^*$).

Let $X \subset S$
be an embedded  submanifold of dimension $n$.
The number of isolated common zeros
of the system of functions
$s^*$ on $X$ is denoted by $N_X(s^*)$. By definition, the average number of isolated common zeros of
all such systems is the integral
$${\mathfrak M}_X = \int _{S^*}\, N_X(s^*)\,ds^*.$$ 
\noindent
We now state a theorem showing that the intergral exists and
computing its value.
For $x = (x_1, \ldots, x_n) \in S$ consider the tangent spaces
$T_i = T_{x_i}(S_i) =\{\xi_i \in E_i\ \vert  \langle \xi_i,x_i\rangle_i = 0\}\subset E_i$.
For
$\xi = (\xi_1,\ldots,\xi_n) \in T = T_1\oplus \ldots \oplus T_n$, write $g_i(\xi_i) = \langle\xi_i,\xi_i\rangle_i$ 
and ${\rm vol }_{1,i}(\xi) =\sqrt{g_i(\xi_i)}$. Then ${\rm vol}_{1,i}$ are 1-densities on $S$.
We recall that the product of 1-densities is defined in \ref{normal}
and use the notations introduced there.

\begin{theorem}\label{prod}{\rm(Crofton formula for the product of spheres).}
$${\mathfrak M}_X = {1\over {\pi^n}}\int_X {\rm vol}_{1 ,1}\cdot \ldots \cdot {\rm vol}_{1,n}.$$
\end{theorem}

\noindent
In what follows, we use some standard notations. Namely, $\sigma _p$
denotes the volume of the $p$-dimensional unit sphere and ${\rm v}_q$
the volume of the $q$-dimensional unit ball.
\begin{example} {\rm Let $C_i \subset S_i,\ i=1,\ldots,n,$ be a circle with center 0 and let
$X = C_1\times \ldots \times C_n \subset S_1\times \ldots \times S_n$ be the  
$n$-dimensional torus. 
For $x=(x_1,\ldots,x_n) \in X$ and $\xi=(\xi_1,\ldots,\xi_n)\in T_xX$ we have ${\rm vol}_{1,i}(\xi) =\vert \xi _i \vert.$
Observe that for $n$ pairwise orthogonal segments $A_1, \ldots, A_n$ of length 1 in ${\mathbb R}^n$
one has ${\rm V}(A_1,\ldots,A_n) =1/ n!$, hence 
$d_1(A_1)\cdot \ldots \cdot d_1(A_n) $ is the Euclidean volume density by Theorem \ref{A_1,...,A_k}.
Apply this to ${\mathbb R}^n = T_xX$ and take $A_i $ tangent to $C_i$, so that
$d_1(A_i)(\xi) =\vert \xi _i \vert$. 
It follows
that the restriction of the $n$-density ${\rm vol}_{1,i}\cdot \ldots \cdot {\rm vol}_{1,n}$ to $X$
is the Riemannian volume density. Therefore
Theorem \ref{prod} implies
$${\mathfrak M}_X = {1\over{\pi^n}}\int _X {\rm vol}_{1,1}\cdot \ldots \cdot {\rm vol}_{1,n}
= 2^n.$$}
\end{example}
\begin{example}{\rm (Crofton formula for the sphere).}
{\rm Let $g$ be the Euclidean metric on $E={\mathbb R}^N$, 
$B$ the unit ball in $E$, $S =\partial B$. The $k$-density on $S$ defined by the induced metric is denoted by ${\rm vol}_{k,g}$.
Put $E_1 =\ldots =E_n = E$ and consider a manifold $X\subset S$, ${\rm dim}\,X =n $, 
embedded diagonally in the product of $n$ spheres $S\times \ldots \times S$.
The classical Crofton formula for the sphere 
$${\mathfrak M}_X ={2\over \sigma _n}\int_X{\rm vol}_{n,g}$$
follows from Theorem \ref {prod}.
Indeed, restricting all densities to $X$ we get
$ {\rm vol}_{1,i} = {\rm vol}_{1,g} = {1\over 2}d_1(B)$ for all $i$,
where $d_1(B)$ is a $1$-density in $E$. 
Hence, from Theorem \ref {A_1,...,A_k} we obtain
$${1\over {\pi ^n}}( {\rm vol}_{1,1}\cdot \ldots \cdot {\rm vol}_{1,n}) =
{d_1(B)^n \over (2\pi)^n}
={n!d_n(B) \over (2\pi)^n} = {n!{\rm v}_n\over
(2\pi)^n}
{\rm vol}_{n,g} 
= {2\over \sigma _n}{\rm vol}_{n,g},$$
where the last equality follows from the relations
$${\rm v}_n\sigma _n =
 {\pi ^{n/2}\over \Gamma ({n\over 2} +1)} \cdot (n+1)
{\pi ^{{n+1}\over2}\over
 \Gamma ({{n+1}\over 2 }+1)}= {2(2\pi)^n \over n!}.$$ }
\end{example}
\bigskip

\noindent
The proof of Theorem \ref{prod} requires some preparations.
In integral geometry, there is a standard technique producing an $n$-density $\Omega $ on $S$, such that
$${\mathfrak M}_X = \int _X \Omega$$
for any $X \subset S$. Namely, $\Omega $ is obtained from $ds^*$ by the pull-back and push-forward operations
in a certain double fibration
$$\begin{matrix}
 & &\ \Gamma & &\cr
             &    & ^{\pi_1 }\swarrow \  \   \searrow  ^{\pi_2} & &\cr
                               && \ \ S\hspace{30pt} S^*, &&\cr
\end{matrix}$$
see \cite {PF}, \cite{GS}, \cite{Ka2}, \cite{Sh}.

Let $\Gamma \subset S\times S^*$
be the submanifold defined by the equations
$x^*_i(x_i) = 0,\  i=1,\ldots,n,$  where $x=(x_1,\ldots,x_n) \in S, \,x^*=(x_1^*,\ldots,x_n^*)\in S^*$.
The projection mappings $S\times S^* \to S$ and $S\times S^* \to S^*$, restricted to $\Gamma$,
are denoted by $\pi_1$ and $\pi_2$, respectively. With these notations, one has the following
proposition.
\begin{proposition}\label {domik}
Let $\Omega =\pi_{1*}\pi_2^*(ds^*)$. Then $${\mathfrak M}_X = \int_X \Omega.$$
\end{proposition}
\begin{proof}
See \cite{PF}, \cite{GS}.
\end{proof}
\noindent
From now on we identify $E_i$ with $E_i^*$ using given scalar products $\langle. , .\rangle_i$.
As a consequence, $S_i$ is identified with $S_i^*$, 
$E$ with $E^*$ and $S$ with $S^*$, respectively. Furthermore, 
$$\Gamma =\{(x,x^*)\in S\times S\ \vert \ \langle x_i,x_i^*\rangle_i \,= 0,\quad  i = 1,\ldots, n\}.$$
Let $\xi_i, \eta _i\in E_i$ be the components of $\xi, \eta \in E$,
respectively.

\begin{lemma}\label{tangent} If $\xi _i \in T_i$ and 
$\eta_i = -\langle\xi_i, x_i^*\rangle_i x_i $
for all $i,\ i=1,\ldots,n$, then
$(\xi, \eta) \in T_{(x,x^*)}\Gamma $.
\end{lemma}
\begin{proof} By the definition of $\Gamma $, the subspace $T_{(x,x^*)}\Gamma  \subset E\oplus E$ is given by
the equations
$$\langle\xi_i,x_i^*\rangle_i + \langle x_i,\eta_i\rangle_i\ =\  
 \langle\xi_i,x_i\rangle_i  \ = \ \langle \eta _i, x_i^*\rangle_i =0 $$
for all $i,\ i=1,\ldots,n$. If $\xi _i \in T_i$ and $\eta _i =-\langle\xi_i,x_i^*\rangle_ix_i$
for all $i$,
then these equations hold true.  
\end{proof}
\noindent
The vector space $T = T_x S= T_1\oplus\ldots\oplus T_n$ is equipped
with the scalar product
$\langle.\,,.\rangle_1 + \ldots + \langle.\, ,.\rangle_n$. For each $i$ denote by $W_i$ the unit sphere with center $0$ in $T_i$,
put $W=W_1\times \ldots \times W_n$ and take a point $w = (w_1,\ldots,w_n) \in W$. Let $\tilde w_i$
be an element of unit length in the highest component of the exterior algebra of $T_{w_i}W_i$, i.e., the
dual to a volume form of $W_i$.
\begin{lemma}\label{R-density} $i_{x_1\wedge\ldots\wedge x_n}ds^*$ is the density of Riemannian volume of $W$ divided by
${\rm vol}(S)$.
\end{lemma}
\begin{proof} Note that $ds_i^*(x_i\wedge \tilde w_i) = 1$. Therefore, by the definition of $ds^*$,
we have
$ds^*(x_1\wedge \ldots \wedge x_n \wedge \tilde w_1 \ldots \wedge \tilde w_n) 
= {1/ {\rm vol}(S)}.$
\end{proof}
\noindent
In the following proposition, we compute the value of density $\Omega $ on the wedge product
$\theta _1 \wedge \ldots \wedge \theta_n$, where $\theta _1, \ldots ,\theta _n \in V$.
As a warning, we remark
that $\theta _i$ are not necessarily contained in $T_i$.

\begin{proposition}\label{omega-value} Let $\Pi _\theta$ be the parallelotope generated by $\theta _i.$
Take the image of $\Pi_\theta$ under the projection map $\pi _w : T \to {\mathbb R} w_1 +\ldots + {\mathbb R} w_n$ 
and denote by ${\rm vol}_w \Pi_\theta $ the volume of $\pi_w(\Pi _\theta )$. Then
$$\Omega(\theta _1 \wedge \ldots \wedge \theta _n) = {1 \over {\rm vol}(S)}\int_W
{\rm vol}_w(\Pi _\theta)\cdot dw_1\cdot\ldots\cdot dw_n,$$
where $dw_i$ is the Euclidean volume form on the sphere $W_i$.   
\end{proposition}

\begin{proof} Let 
$\theta _i = (\theta _{i1},\ldots, \theta_{in})$, where $\theta_{ij} \in T_j$ and 
put
$$\eta _{x,w,\theta_i} = -\sum _{j=1}^n \langle \theta _{ij} ,w_j\rangle_j \cdot x_j.$$
Then
$$h_i =(\theta _i, \eta _{x,w,\theta_i}) \in T_{(x,w)} \Gamma $$
by Lemma \ref{tangent}. Recall that $\Omega = \pi_{1*}\pi_2^*(ds^*)$.
Since $\pi_1^{-1}(x) = \{x\}\times W$, we have
$$\Omega (\theta _1 \wedge \ldots \wedge \theta_n) = \int _{\{x\}\times W}
i_{h_1\wedge \ldots \wedge h_n} \pi_2^*(ds^*) $$
by the definition of push-forward. The projection map $\pi_2$ 
sends $h_i$  to $\eta_{x,w,\theta _i}$ and defines a diffeomorphism between $ \pi_1^{-1}(x)$ and $W$. From the previous equality
it follows that 
$$\Omega (\theta _1 \wedge \ldots \wedge \theta_n) = \int _ W
i_{\eta_{x,w,\theta_1}\wedge \ldots \wedge \eta_{x,w,\theta_n}}\, ds^*.$$
Let $A$ be the matrix with entries $a_{ij} = \langle\theta_{ij}, w_j\rangle_j$. The definition of $\eta_{x,w,\theta _i}$
shows that
$$\eta _{x,w,\theta _1}\wedge \ldots \wedge 
\eta_{x,w,\theta _n} = \pm\, {\rm det}(A)\, x_1\wedge \ldots \wedge x_n
= \pm\, {\rm vol}_w(\Pi_\theta )\, x_1 \wedge \ldots \wedge x_n.$$
The density $ds^*$ is non-negative, so the above formula for $\Omega (\theta_1\wedge \ldots \wedge \theta_n)$ yields
$$\Omega (\theta _1 \wedge \ldots \wedge \theta _n) = \int_W{\rm vol}_w(\Pi _\theta)\cdot i_{x_1\wedge
\ldots \wedge x_n}ds^* .$$
Together with Lemma \ref{R-density} this completes the proof. 
\end{proof}
\noindent
{\it Proof of Theorem  \ref{prod}}. Let $D \subset T$ be a compact convex set of dimension $n$.
Proposition \ref{omega-value} gives an expression for the value of $\Omega $ on $D$. 
If $m_i = {\rm dim}\,T_i$ then
\begin{equation}\label{*}
\Omega (D) = {1\over \prod _i\sigma_{m_i}}\cdot \int_{W_1\times \ldots \times W_n}{\rm vol}_w(D)\cdot dw_1\cdot
\ldots \cdot dw_n.
\end{equation}
We have to prove
that \begin{equation}\label{**}
\Omega (D) = {1\over \pi^n} \,({\rm vol}_{1,1}\cdot \ldots \cdot{\rm vol}_{1,n})(D).
\end{equation}
The density ${\rm vol}_{1,i}$ is the pull-back of the Riemannian 1-density ${\rm vol}_{1,g_i}$ on $T_i$ under the 
projection map $\pi_i: T \to T_i$.  According to (\ref{k-density}), the gauge function $\phi $ on ${\rm Gr}(1,T_i)$, associated with
 ${\rm vol}_{1,g_i}$, equals 1. If $\Phi$ is the preimage of $\phi $ under
the cosine transform then 
$$ \Phi = {\sigma _{m_i-1}\over 2{\rm v}_{m_i-1}}$$
by a direct calculation.
Applying Proposition \ref{cos-transf}, we get ${\rm vol}_{1,g_i} =\chi _1(\mu _{g_i})$
for the normal measure $\mu _{g_i}= \mu _{1,\Phi}$ on ${\rm Gr}_a(1,T_i)$ defined by ({\ref {affine-to-vector}).
Namely, for a subset $A \subset {\rm Gr}_a(1,T_i)$ and for $H \in {\rm Gr}(1,T_i)$ put
$H_A = \{h \in H \ \vert \  h+H^\perp \in A\}$.
Then
\begin{equation}\label{mu}
\mu _{g_i}(A) = {\sigma _{m_i - 1}\over 2{\rm v}_{m_i-1}}\ \int _{{\rm Gr}(1,T_i) }\lambda _i(H_A)\,dH ,
\end{equation}
where $\lambda _i$ is the Lebesgue measure corresponding to $g_i$ on the line $H$.

Assume first that $n = 1$, so that ${\rm vol}_{1,1} = {\rm vol}_{1,g_1}$. If $A \subset {\rm Gr}_a(1,T)$
is formed by affine hypersurfaces intersecting $D$, then

$$\Omega (D) = {1\over \sigma_{m_1}}\int_{W_1}{\rm vol }_w(D)\, dw =
{\sigma _{m_1-1}\over \sigma _{m_1}} \int _{{\rm Gr}(1,T_1)}\lambda _1(H_A)\,dH=$$

$$= {\sigma_{m_1-1}\over \sigma _{m_1}}\cdot {2{\rm v}_{m_1-1}\over \sigma_{m_1-1}}\cdot  \mu_{g_1}(A)={\mu_{g_1}(A)\over \pi} = {{\rm vol}_{1,1}(D)\over\pi}.$$

\noindent For $n = 1$ the theorem is proved. 

Assume now that $n > 1$ and let $\mu_i = \pi_i^*(\mu_{g_i})$. By
Proposition \ref{n-pull-back} we have ${\rm vol}_{1,i} = \chi_1(\mu _i)$. Hence
$${\rm vol}_{1,1} \cdot \ldots \cdot {\rm vol}_{1,n} = \chi_n(\mu_1\cdot \ldots \cdot \mu_n)$$
by the definition of
the product of normal densities. Therefore
$$({\rm vol}_{1,1} \cdot \ldots \cdot {\rm vol}_{1,n})(D) = (\mu_1\cdot \ldots \cdot \mu_n)({\mathcal J}_n(D)),$$
where
$$  {\mathcal J} _n(D)=\{(H_1, \ldots, H_n)\in ({\rm Gr}_a(1,T))^n\ \vert \ D \cap H_1 \cap \ldots \cap H_n
\ne \emptyset\}.$$
The support of $\mu _i$ is the set ${\mathcal G}_i$ of affine  hyperplanes containing affine shifts of $
\oplus_{j\ne i}T_j$. For any set $A\subset{\mathcal G}_i$ its measure $\mu_i(A)$ is equal to the value of $\mu_{g_i}$
on the projection of $A$ onto ${\rm Gr}_a(1,T_i)$. The product measure
$\mu_1\cdot \ldots \cdot \mu_n$ is supported on the surface ${\mathcal G}\subset {\rm Gr}_a(n,T)$
formed by affine subspaces $G_1\cap \ldots \cap G_n$, where $G_i \in {\mathcal G}_i$.
For $G \in {\rm Gr}_a(1,T_n)$ put $\bar G =\pi _n^{-1}(G)$. Let
$${\mathcal I}(D \cap \bar G) = \{(G_1,\ldots G_{n-1}) \ \vert \  G_i \in {\mathcal G}_i,
D \cap \bar G \cap G_1 \cap \ldots \cap G_{n-1} \ne \emptyset \}.$$
Then
$$(\mu _1\cdot \ldots \cdot \mu _n)({\mathcal J}_n(D))= \int _{G \in {\rm Gr}_a(1, T_n)} 
(\mu _1 \cdot \ldots \cdot \mu _{n-1})({\mathcal I}(D \cap \bar G))\, d\mu_{g_n}$$
by Fubini's theorem.
Let $U = T_1+\ldots+T_{n-1} $ and let $\pi _U : T \to U$ be the projection map.
We use the same notation $\mu_1\cdot \ldots \cdot \mu_{n-1}$ for the product of measures $\mu_i$
on ${\rm Gr}_a(n-1, U)$ and for its pull-back under $\pi _U$ on ${\rm Gr}_a(n-1,T)$. Keeping this in mind, we
get
$$(\mu_1\cdot \ldots \cdot \mu _{n-1})({\mathcal I}(D\cap \bar G)) =
(\mu_1\cdot \ldots \cdot \mu _{n-1})({\mathcal J}_{n-1}(\pi _U(D\cap \bar G))).$$
Applying the definition of $\Omega $ to $U$, we get an $(n-1)$-density,
to be denoted by $\Omega _U$. For $H \in {\rm Gr}(1,T_n)$ we denote by $H^\perp $
the orthogonal complement to $H$ in $T$. By induction and by (\ref{mu}) for $i=n $ we have
$$
({\rm vol}_{1,1}\cdot \ldots \cdot {\rm vol}_{1,n})(D)=
 \int _{G \in {\rm Gr}_a(1,T_n)}({\rm vol}_{1,1}\cdot \ldots \cdot{\rm vol}_{1,n-1})
(\pi _U(D\cap \bar G))\, d \mu_{g_n} =
$$
$$
= {\pi}^{n-1}\int _{G\in {\rm Gr}_a(1,T_n)}\, \Omega _U(\pi _U(D\cap\bar G))\, d\mu_{g_n}.$$
Finally, using the expression (\ref{mu}) for $\mu _{g_n}$, we obtain
\begin{equation}\label{integrals}
({\rm vol}_{1,1}\cdot \ldots \cdot {\rm vol}_{1,n})(D)=
\end{equation}
$$=
{{\sigma _{m_n-1}}\pi ^{n-1}\over 2{\rm v}_{m_n-1}}\int _{H \in {\rm Gr}(1,T_n)}
\int _{h\in H}
\Omega _U(\pi _U(D\cap (h+H^\perp))\,dh\, dH.$$
On the other hand, by Fubini's theorem we can write (\ref{*}) in the form
$$\Omega (D) = {1\over\prod _{i=1}^n \sigma_{m_i}} \cdot
\int _{W_n}F(w_n)\ dw_n,$$
where $$F(w_n) = \int _{W_1\times \ldots \times W_{n-1}} {\rm vol}_w(D) \, dw_1 \cdot \ldots \cdot w_{n-1}.$$
Let $\omega $ and $\omega ^\prime $ be the subspaces generated by $w_1, \ldots, w_{n-1}, w_n$ and $w_1, \ldots, w_{n-1}$, respectively. Denote by $H$ the line generated by $w_n$. Then
$${\rm vol}_w(D) = \int _H{\rm V}_{n-1}(\pi _\omega (D)\cap(h+H^\perp))\,dh.$$
Plug this in the integral $F(w_n)$ and observe that
$$\pi _\omega(D)\cap(h+H^\perp) = h+\pi_{\omega ^\prime}(D \cap (h+H^\perp))
= h+\pi_{\omega ^\prime}\pi_U(D \cap (h+H^\perp)).$$
As a result, we get
$$F(w_n) = \int_H \int _{W_1\times\ldots\times W_{n-1}}{\rm V}_{n-1}(\pi _{\omega^\prime}\pi_U(D\cap
(h+H^\perp)))=
$$
$$=\prod_{i=1}^{n-1}\sigma_{m_i} \cdot \int _H\Omega _U(\pi _U(D\cap(h+H^\perp))\,dh$$
by the definition of $\Omega _U$. Therefore
$$\Omega (D) = {1\over \sigma_{m_n}} \int_{W_n}\int _H\Omega _U(\pi _U(D\cap(h+H^\perp))\, dh\, dw_n.$$
Passing from the sphere to the projective space, we obtain
$$\Omega (D) = {\sigma _{m_n-1}\over \sigma_{m_n}} 
\int_{{\rm Gr}(1,T_n)}\int _H\Omega _U(\pi _U(D\cap(h+H^\perp))\, dh\, dH.
$$
Since
$\sigma _{m_n} = 2\pi \cdot {\rm v}_{m_{n-1}}$, equality (\ref {**}) follows from (\ref{integrals}).
\hfill $\square$

\subsection
{ Proof of Theorem \ref{first}.}\label{proof1}
We use the notations introduced in \ref{Crofton} with $E_k = V_k^*$. Then  $S_k \subset V_k^* $, $S_k^* \subset V_k$, and
$g_k$ is the pull-back of the metric form on $V_i$ under $\theta _k : X \to S_k ^*\subset V_k$.
For $s^* = (s_1^*, \ldots, s_n^*) \in S^* = S_1^* \times \ldots \times S_n^*$
we denote by $N_U(s_1^*, \ldots, s_n^*)$ the number of isolated common zeros
of $s_1^*, \ldots, s_n^*$ in an open set $U \subset X$.
We want to prove that $N_X(s_1^*, \ldots, s_n^*)$ is integrable and
compute the integral whose value divided by $\sigma $ is denoted by ${\mathfrak M}_X(V_1,\ldots,V_n)$, see \ref{average}.
Theorem \ref{first} is a consequence of the equality
\begin{equation}\label{main}
{\mathfrak M}_X(V_1, \ldots, V_n) = {1\over {\pi ^n}} \int_X {\rm vol}_{1,g_1}\cdot \ldots \cdot {\rm vol}_{1,g_n}.
\end{equation}
Indeed, if ${\mathcal E}_i$ is the ellipsoid corresponding to $g_i$ then for all $x \in X, \xi \in T_x(X)$ one has
$${\rm vol}_{1,g_i}(\xi ) = {d_1({\mathcal E}_i(x))(\xi ) \over 2},$$
where $d_1({\mathcal E}_i(x))$ is considered as a translation invariant 1-density.
By Theorem \ref{A_1,...,A_k} we have
$${1\over {\pi }^n}{\rm vol}_{1,g_1}\cdot \ldots \cdot {\rm vol}_{1,g_n} =
{n!\over (2\pi)^n} d_n({\mathcal E}_1(x), \ldots , {\mathcal E}_n(x)).$$
By the definition of $d_n$ it follows that
$${1\over {\pi }^n}({\rm vol}_{1,g_1}\cdot \ldots \cdot {\rm vol}_{1,g_n})(\xi _1 \wedge \ldots \wedge \xi _n) =
{n!\over (2\pi)^n} {\rm V}_n({\mathcal E}_1(x), \ldots , {\mathcal E}_n(x)){\rm vol}_n(\Pi_\xi),$$
where ${\rm vol}_n $ is taken with respect to $g$ on $T_x(X)$ and the mixed volume
${\rm V}_n$ with respect to $g^*$ on $T_x^*(X)$. Combining this equality with (\ref{m-density}), we get 
Theorem \ref {first} from (\ref{main}).

Now, if the mapping $\theta = \theta _1 \times \ldots \times \theta _n$ is an embedding of $X$ into $S_1^* \times 
\ldots \times S_n^*$ then (\ref {main}) is a direct consequence of
Crofton formula (Theorem \ref {prod}) and equality (\ref{m-density}). In the general case, it is enough to prove
(\ref{main}) for an arbitrary relatively compact domain $U \subset X$. In fact, $X$ is the union of an increasing
sequence of such 
domains $U_i$.
Thus, if 
\begin{equation} \label{main1}
{\mathfrak M}_{U}(V_1, \ldots, V_n) = {1\over {\pi ^n}} \int_{U} {\rm vol}_{1,g_1}\cdot \ldots \cdot {\rm vol}_{1,g_n}
\end{equation}
for each $U $ then we get two increasing sequences on the both sides of
(\ref{main1}) for $U = U_i$ and (\ref{main}) is obtained 
by passing to the limit. We will now prove (\ref {main1}). Let $D$ be the set of critical points of $\theta $ in the 
closure $\bar U$ of a relatively compact domain $ U \subset X$. If $D\ne \emptyset$ then take a decreasing sequence of
relatively compact neighborhoods $D_i$ of $D$, such that $\cap\, D_i = D$.

\smallskip
\noindent
To finish the proof, we need three lemmas.
\begin{lemma} \label{final1}
If $\theta $ has no critical points in $U$ then (\ref {main1}) holds true.
\end{lemma}
\begin{proof}
Assume first that $D = \emptyset $.
Then there is a finite covering
$\{U_i\}$ of $U$ such that $\theta \vert _{U_i}$ is a closed embedding for every $i$.
We know that (\ref{main1}) is valid for all intersections $U_{i_1}\cap \ldots \cap U_{i_k}$. Since
the both parts are additive functions of a domain, (\ref{main1}) for the union of $U_i$ follows by inclusion-exclusion principle.
In case $D\ne \emptyset$ one can 
apply (\ref {main1}) to $U \setminus D_i$ and get the required assertion as $i \to \infty$. 
 \end{proof}

\begin{lemma}\label{final2}
One has
$${\rm lim}_{i\to \infty}\int_{D_i}{\rm vol}_{1,g_1}\cdot \ldots \cdot {\rm vol}_{1,g_n} = 0.$$
\end{lemma}
\begin{proof}
Choosing a Riemannian metric $g$ on $X$ we can write
$$({\rm vol}_{1,g_1}\cdot \ldots \cdot {\rm vol}_{1,g_n})(\xi _1 \wedge \ldots \wedge \xi_n) = f(x) \cdot {\rm vol}_{n,g}(\xi_1\wedge\ldots \wedge \xi_n), $$
where $\xi _ i \in T_x$. Then $f(x)$ is a non-negative continuous function and
${\rm max}_{x\in \bar {D_i}}f(x) \to 0$ as $i\to \infty$, hence the assertion.
\end{proof}
\noindent
For 
an arbitrary subset $A_k \subset S_k$ denote by $A_k^* \subset S^*_k$ the set of linear functions with a zero in $A_k$.
Consider $s^* = (s_1^*, \ldots, s_n^*) \in S^*$ as an $n$-tuple of functions on $S$.
For $A \subset S$
put $A^* =\{s^* \in S^* \, \vert \, \exists x\in A:   \forall i \  s_i^*(x) = 0\}$.
Note that if $A = A_1\times \ldots \times A_n$ then $A^* = A^*_1 \times \ldots \times A^*_n$.

\begin{lemma}\label{final3}
Let ${\mathcal D} \subset S$ be the set of critical values of $\theta $ on $\bar U$.
Then ${\mathcal D}^* \subset S^*$ has measure 0.
\end{lemma}\begin{proof}
By Sard's lemma the $n$-dimensional Hausdorff measure of $\mathcal D$ is 0. 
Consider a covering of ${\mathcal D}$ by open sets $B_i =B_{i,1}\times \ldots \times B_{i,n}$,
where $B_{i,k} \subset S_k$ are open balls of the same radius $r(B_i)$ for any given $i$.
For $\epsilon >0 $ small enough there exists such a covering with $r(B_i) < \epsilon $ and 
$\sum _ir^n(B_i)$ arbitrarily small. Let ${\mathcal D}^*(\epsilon) = \cup _iB_i^*$ for such a covering.
Then $B_i^* = B_{i,1}^*\times \ldots \times B_{i,n}^*$ by construction, hence
$${\rm vol}_{S^*}(B_i^*) = \prod_{k=1}^n{\rm vol }_{S_k^*} B_{i,k}^*  < Cr^n(B_i).$$
By the definition of Hausdorff measure $\sum r^n(B_i) \to 0$ as $\epsilon \to 0$.
Therefore ${\rm vol}({\mathcal D}^*(\epsilon)) \to 0$.
\end{proof}
\noindent
{\it End of the proof.} Let $U$ be a relatively compact domain in $X$. Then
$N_U(s_1^*, \ldots, s_n^*) = N_{U\setminus D}(s_1^*, \ldots, s_n^*)$ 
almost everywhere on $S^*$ by Lemma \ref{final3}. But $N_{U\setminus D}$ is integrable by Lemma \ref{final1}.
Thus $N_U$ is also integrable
and the integrals of $N_U$ and $N_{U \setminus D}$ are equal. Finally, by Lemma \ref {final2} one can replace 
the integration domain $U$ on the right hand side of (\ref{main1}) by $U \setminus D$.
\hfill $\square $

\end {document}